\documentclass[11pt]{amsart}
\usepackage{fullpage, style}
\usepackage{subfiles}
\usepackage{comment}
\newtheorem*{rep@theorem}{\rep@title}
\newcommand{\newreptheorem}[2]{%
\newenvironment{rep#1}[1]{%
 \def\rep@title{#2 \ref{##1}}%
 \begin{rep@theorem}}%
 {\end{rep@theorem}}}
\makeatother

\newreptheorem{theorem}{Theorem}
\newreptheorem{lemma}{Lemma}
\newreptheorem{conjecture}{Conjecture}

\title{Hyperbolicity of Staked Links and Lower Bounds on Their Volumes}

\author[C. Adams]{Colin Adams}
\address{Department of Mathematics, Williams College, Williamstown, MA 01267}
\email{cadams@williams.edu}

\author [F. Gomez-Paz]{Francisco Gomez-Paz}
\address{Department of Mathematics, MIT,  77 Massachusetts Avenue, Cambridge, MA 02139-4307} 
\email{pjgomez@mit.edu}

\author[J. Kang]{Jiachen Kang}
\address{Mathematics Department, Harvard University, 1 Oxford St, Cambridge, MA 02138}
\email{jkang@math.harvard.edu}

\author[L. Krause]{Lukas Krause}
\address{Schaffhauserstrasse 455, postal code 8052, Zurich, Switzerland}
\email{lukrau2002@gmail.com}

\begin{document}

\begin{abstract}
    We define a class of links in handlebodies called ``charm bracelets," which are a subset of staked links. We provide tools to construct infinitely many such hyperbolic links and bound the corresponding volumes from below in terms of volumes corresponding to the individual charms.
\end{abstract}

\maketitle

\section{Introduction}

 A staked link diagram is an immersion of a finite collection of circles in a projection surface with a finite number of double points, each with a choice of over or under for the strands at the crossing,  together with a choice of finitely many points on that surface that avoid the immersion, called ``stakes," with the projection considered up to isotopy and Reidemeister moves away from the stakes. Staked links were introduced in full generality in \cite{knotoids} as a subset of generalized knotoidal graphs, extending the class of ``tunnel links'' considered in \cite{kauffmanstaked}. 

In \cite{knotoids}, staked knots were realized topologically as knots in handlebodies as follows. Thicken the projection surface $F$ into $F \times I$. Then, the immersion becomes a link in $F \times I$. For each stake $s\in F$, remove the interior of a regular neighborhood of $s\times I$ from $F\times I$. The result of this process is the complement of a link in a handlebody. See Figure \ref{fig:handlebodyrealization} for an example. Note that every knot or link in a handlebody of genus at least one can be realized as a staked link.

\begin{figure}[htbp]
    \centering
    \includegraphics[scale=0.4]{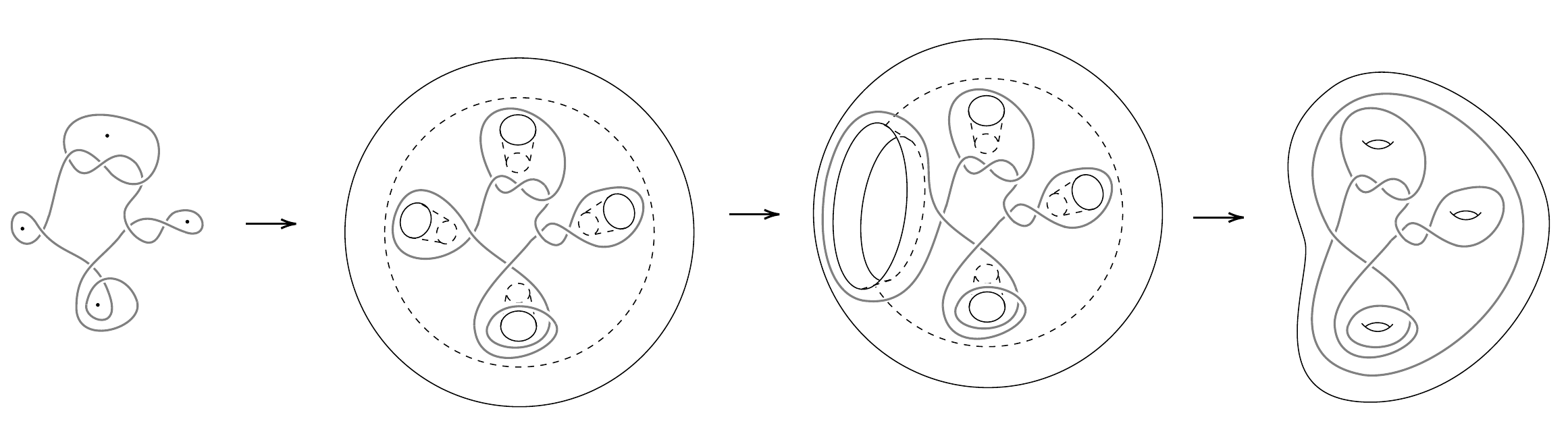}
    \caption{Realizing a staked knot as a knot in a handlebody.}
    \label{fig:handlebodyrealization}
\end{figure}

We say that a staked link is \textit{hyperbolic} if its complement in the handlebody given by the above construction possesses a hyperbolic metric with totally geodesic boundary. The Mostow-Prasad Rigidity Theorem asserts that if such a metric exists, it must be unique.

There has been previous work on hyperbolic staked links. In \cite{adams2023hyperbolicity}, the authors  provide necessary and sufficient conditions for an alternating staked link to be hyperbolic. They also provide a characterization of a family of non-alternating hyperbolic staked links (Corollary 5.1 of that paper). Extending methods from \cite{adams2021lower} and applying results that will appear in a companion paper \cite{AdamsKang}, we are able to obtain a much larger class of hyperbolic non-alternating staked knots and links, by considering a large class of staked links we call \textit{charm bracelets}.

A \textit{charm} is a part of a staked link diagram on the sphere that is contained in a disk on the sphere that has boundary avoiding crossings and stakes, such that the boundary is punctured at least twice by the link and the disk contains at least one crossing and one stake. A charm bracelet is any staked link diagram that can be partitioned into charms, the corresponding disks for which all intersect in a boundary point $t$ at the top and a boundary point $b$ at the bottom such that for each charm, at least one strand of the tangle ends on the left boundary of the disk and one on the right boundary. See Figure \ref{fig:charm}.

\begin{figure}[htbp]
    \centering
    \includegraphics[scale=0.6]{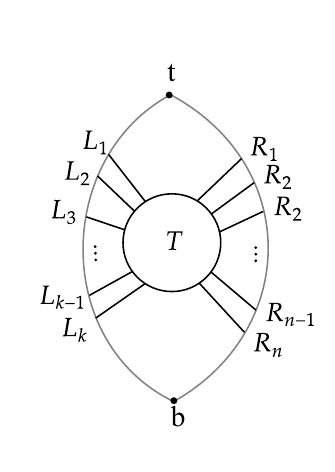}
    \caption{A charm $C$ from a charm bracelet.}
    \label{fig:charm}
\end{figure}

Given a charm $C$,  we can reflect $C$ across a vertical axis through the top and bottom boundary points in the projection plane to obtain $C^R$. Then, we can define the $2m$-replicant of $C$ to be the charm bracelet $B$ given by concatenating cyclically the $2m$ charms in order  $B = (C, C^R, C, C^R, \dots, C, C^R)$. In the case $B$ is hyperbolic, we say $C$ is $2m$-hyperbolic and define the $2m$-volume of $C$ to be $vol^{2m}(C) = \frac{1}{2m} vol(B)$.

The advantage to having $2m$-hyperbolic charms comes from the following version of Theorem 3.7 from \cite{adams2021lower} that applies to charm bracelets. 

\begin{theorem} If each charm in an even charm bracelet $B = (C_1, C_2, \dots , C_{2m})$ is $2m$-hyperbolic, then $B$ is hyperbolic and \[
        \vol(B) \geq \sum _{i= 1}^{2m} \vol^{2m}(C_i).
    \]  
\end{theorem} 

As an example of the application of this theorem, see Figure \ref{fig:DecompositionTheoremExample}.

\begin{figure}[htbp]
    \centering
    \includegraphics[scale=0.6]{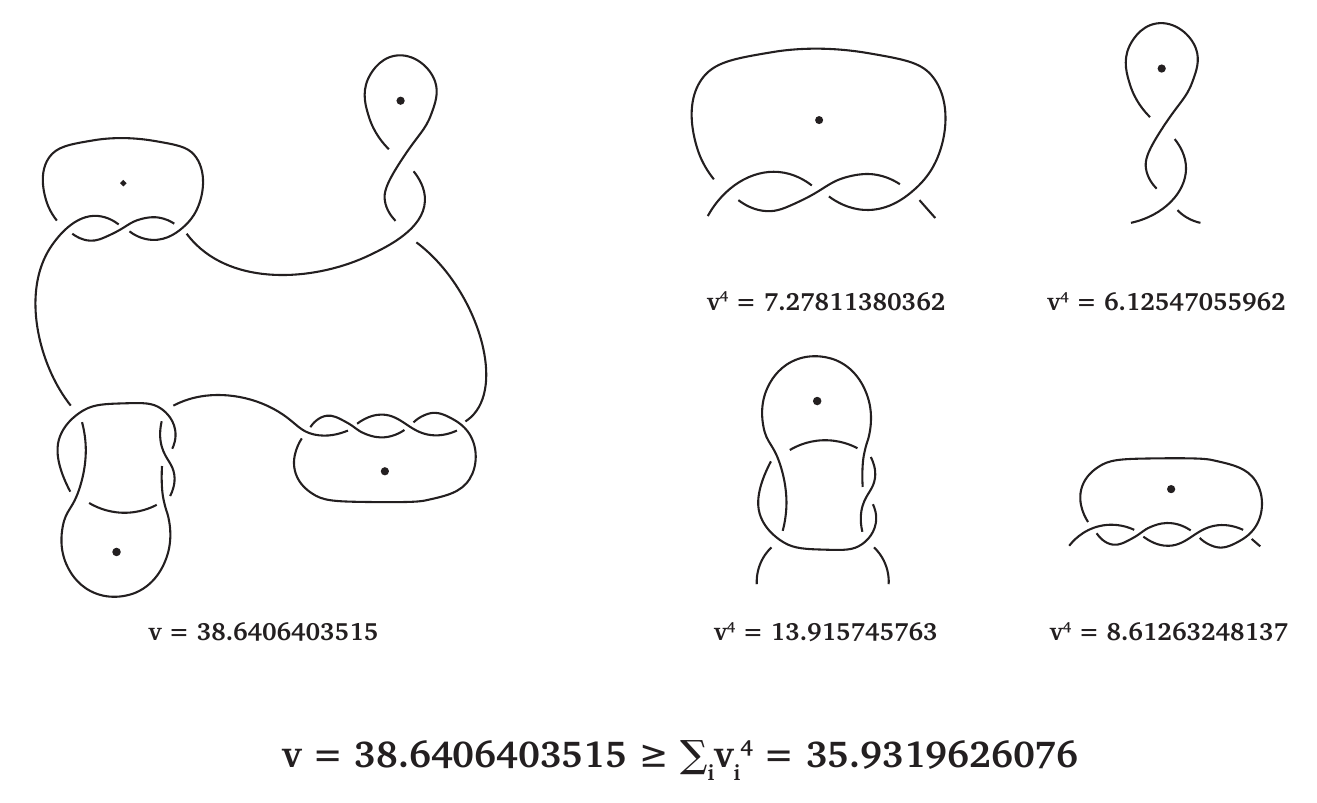}
    \caption{An example of the application of Theorem 1.1.}
    \label{fig:DecompositionTheoremExample}
\end{figure}

Note that there are no restrictions to alternating bracelets or charms. To be able to utilize this theorem, one would like to generate numerous charms that are $2m$-hyperbolic for various values of $m \geq 1$. In Section \ref{sect:prelim}, we provide preliminary definitions. In Section \ref{sect:hyperbolicity}, we show that if any charm is 2-hyperbolic, it is $2m$-hyperbolic for all $m\geq 1$. 

In  a separate paper \cite{AdamsKang}, it is shown that an alternating charm is 2-hyperbolic if and only if it satisfies a certain set of conditions that are easily determined from its diagram. These are then $2m$-hyperbolic for all $m \geq 1$ and can therefore be used to create hyperbolic charm bracelets with any even number of charms. 

This shows that a large variety of bracelets built from charms are hyperbolic, including many non-alternating bracelets, and provides a means to obtain lower bounds on their volumes.  

In previous work (c.f. \cite{Adamsproceedings}), we investigated charm bracelets where each charm was as simple as possible, which is to say a single-stranded charm with one crossing that generated a monogon with a single stake inside it. Such charm bracelets are called bongles, as in 
Figure \ref{fig:bonglesexamples}. 

\begin{figure}[htbp]
    \centering
    \includegraphics[scale=0.6]{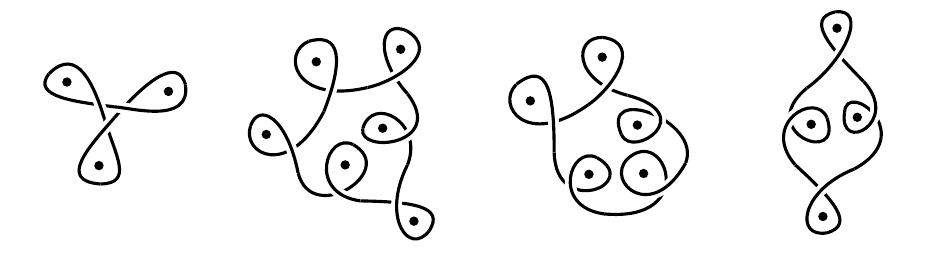}
    \caption{Examples of bongles.}
    \label{fig:bonglesexamples}
\end{figure}
We showed that a bongle is hyperbolic if and only if it is alternating, and determined examples of when hyperbolic bongles had the same volume. Here, we are extending to much more general charm bracelets.
\medskip

\noindent \textbf{Acknowledgments.} \\
This research was supported by Williams College, the Finnerty Fund, and NSF Grant DMS2241623 via the SMALL Undergraduate Research Program. We are grateful to Gregory Li, Reyna Li, Chloe Marple and Ziwei Tan, who are other members of the SMALL 2024 knot theory group, for many helpful discussions and suggestions.

\section{Preliminaries}\label{sect:prelim}

\begin{definition}
    A link $L$ in a compact, orientable 3-manifold $M$ is said to be \textit{tg-hyperbolic} if, after capping off all spherical boundaries with 3-balls and shaving off all torus boundaries, $M\setminus N(L)$ admits a complete hyperbolic metric such that its higher genus boundary components are totally geodesic. (Here and elsewhere, $N(X)$ denotes an open neighborhood of $X$.)
\end{definition}

\begin{definition}
    A \textit{staked link} $L$ in a surface $F$ is an embedding of a finite collection of circles into $F\times I$, together with a choice of a finite nonzero number of points on $F$, called stakes. We may take a projection $\pi(L)$ of $L$ onto $F$ in the usual way, such that $\pi(L)$ avoids the stakes. We consider this projection up to isotopy and Reidemeister moves. At the stakes there is a forbidden fourth move, which is depicted in Figure 1 of \cite{knotoids}, and which prevents strands of the projection from being slid across the stake points. 
\end{definition}

Note that it is equivalent to define a staked link as an embedding of $L$ into $(F\setminus U)\times I$, where $U$ is the disjoint union of small open disk neighborhoods of the points chosen as stakes. In this paper, $F$ will be $S^2$. In the case that $(F\setminus U)\times I \setminus N(L)$ is tg-hyperbolic, its hyperbolic volume is uniquely determined and denoted $\vol(L)$.

\begin{definition}
    Place a disk $D$ on a staked link diagram such that its boundary intersects the link projection transversely in two or more points and does not intersect the stakes. Also make sure $D$ contains at least one stake and at least one crossing. Avoiding $\pi(L) \cap \partial D$, choose two points $t$ and $b$ for top and bottom on $\partial D$, dividing it into two arcs, each of which contains at least one point of $\pi(L) \cap \partial D$. Let the places where the links intersects one of these arcs be called the ``left" endpoints and the places where the link intersects the other be called the ``right" endpoints. The intersection of the disk with the link projection and the stakes, together with a choice of left and right endpoints, is called a \textit{charm}. 
\end{definition}

A charm with $p$ left endpoints and $q$ right endpoints is called a $(p,q)$-stranded charm. If $p=q$ we will simply write that it is $p$-stranded.

We can realize a charm containing $r$ stakes as an embedding of a collection of arcs and circles into a thickened disk with $r$ open solid cylinders removed at the stakes. We also refer to this realization as a charm. The points $t$ and $b$ become vertical line segments on the boundary of the thickened disk. Let $L_C$ denote the restriction of the link $L$ to the charm.

Choosing the point $t$ on the disk boundary as a starting point, order the left and right strands by their proximity to the $t$. If charm $C_1$ has the same number of right endpoints as $C_2$ has left endpoints then we can concatenate them by identifying the right boundary of $C_1$ with the left boundary of $C_2$, identifying right strands of $C_1$ to left strands of $C_2$ in a way that respects their order. This is depicted in Figure \ref{fig:charmconcat}. We say $C_1$ and $C_2$ are concatenatable.

\begin{figure}
    \centering
    \includegraphics[width=0.7\linewidth]{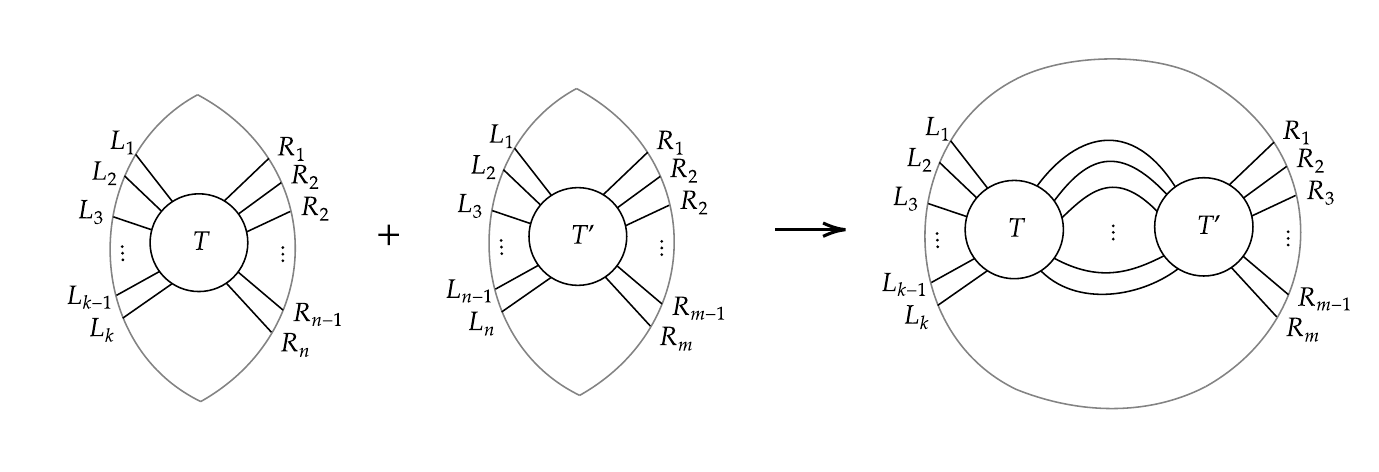}
    \caption{Concatenating two $n$-stranded charms.}
    \label{fig:charmconcat}
\end{figure}

\begin{definition}
    Given a sequence $w = (C_0, \dots ,C_{n-1})$ of charms such that $C_i$ and $C_{i+1}$ are concatenatable (mod n), we may obtain a staked link diagram in $S^2$ by concatenating $C_i$ and $C_{i+1}$ (subscripts interpreted mod $n$). The resulting staked link $L$ is the \textit{charm bracelet} corresponding to $w$. We call $w$ a \textit{bracelet word}.
\end{definition}

\begin{definition}
Given a tg-hyperbolic charm bracelet, the volume of the charm bracelet is the volume of the complement of the bracelet $L$ in the handlebody that is $S^2 \times I$ with open neighborhoods of the stakes removed.  
\end{definition}

\begin{definition}
    A \textit{$2n$-replicant} of a charm $C$ is the bracelet of length $2n$ in that charm and its reflection given by the vector $(C,C^R,C,C^R,\cdots,C,C^R)$. It is denoted $D^{2n}(C) = (H_{2n,C}, L_{2n,C})$, where $H_{2n,C}$ is the resulting handlebody that comes from removing open neighborhoods of the stakes in $S^2 \times I$ and $L_{2n,C}$ is the resulting link in that handlebody.  See Figure \ref{fig:replicant examples} for examples. 

    \begin{figure}
    \centering
    \includegraphics[scale=0.7]{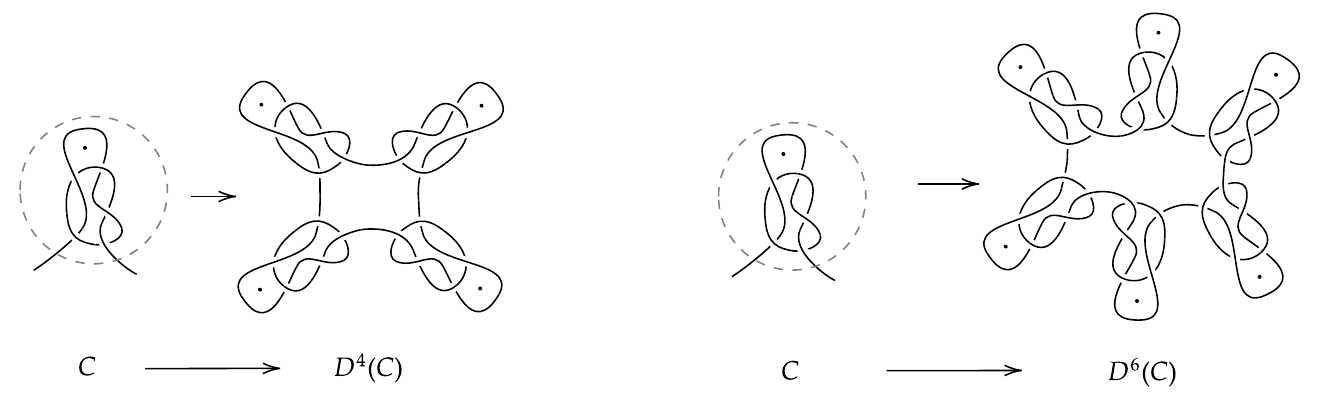}
    \caption{Examples of a $4$-replicant and a $6$-replicant.}
    \label{fig:replicant examples}
\end{figure}
    
    If the complement of $L_{2n,C}$ in $H_{2n,C}$ is tg-hyperbolic, we say $D^{2n}(C)$ is tg-hyperbolic.  We define $\vol(D^{2n}(C)) = \vol(H_{2n,C} \setminus L_{2n,C})$. We then define the \textit{$2n$-volume of a charm $C$} to be  $\vol^{2n}(C) = \frac{1}{2n} \cdot \vol(D^{2n}(C))$. 
\end{definition}

\section{Hyperbolicity of charm bracelets} \label{sect:hyperbolicity}

\subsection{Results}

\begin{theorem}  \label{cor:2_replicant}
    Let $C$ be a charm. If $D^{2}(C)$ is tg-hyperbolic then $D^{2m}(C)$ is tg-hyperbolic for all $m \geq 1$.
   \end{theorem}

   In order to prove this theorem, we first prove another theorem from which this will follow. But, we do need a bit more notation.

\subsection{Notation}

We follow notation introduced in \cite{adams2021lower}.

\begin{definition}Let $C$ be a charm, and $D^{2n}(C) = (H_{2n,C},L_{2n,C})$. For convenience of notation, let $M = H_{2n,C}$, and let $L = L_{2n,C}$.  A \textit{sheet} $\Sigma_i$ is the surface that is the properly embedded fixed point set of a reflection isometry for $D^{2n}(C)$ as shown in Figure \ref{fig:SplittingAndWedges}. Note that for $D^{2n}(C)$ there are $n$ distinct sheets, each an annulus intersected multiple times by the corresponding link. The \textit{separating starburst} $\mathbf{S'} = \cup_i \Sigma_i$ is the union of the set of sheets in $D^{2n}(C)$.  The separating starburst $\mathbf{S'}$, divides $D^{2n}(C)$ into a series of $2n$ homeomorphic 
connected components in the link exterior, called \textit{pieces} and denoted $P_i$. Note that $P_i$ consists of a ball with open neighborhoods of stakes removed, and the open neighborhood of the tangle that is the intersection of $L$ with this ball removed.  A \textit{wedge} $\widehat{\Sigma}_i$ is the intersection of the  separating starburst with a piece $P_i$. Note it lies in $\partial P_i$ and is punctured by the endpoints of the tangle. 
\end{definition}

\begin{figure}[htbp]
    \centering
    \includegraphics[trim = {2cm 2cm 2cm 2cm}, scale=0.2]{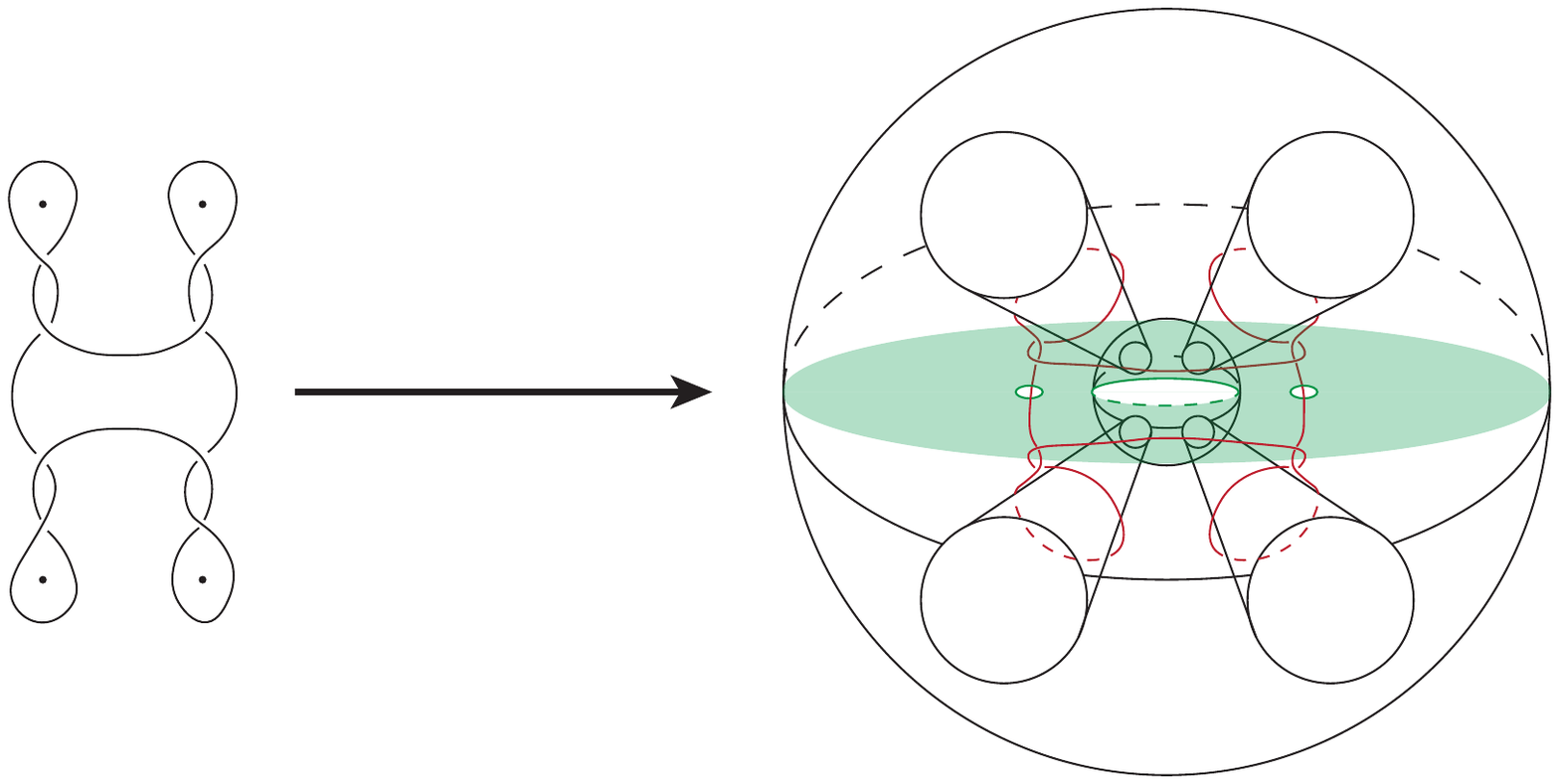}
    \caption{A replicant and a sheet.}
    \label{fig:SplittingAndWedges}
\end{figure}

\begin{definition}Let $\frac{1}{2}D^{2n+2}(C)$ denote the portion of $D^{2n+2}(C)$ lying to one side of $\Sigma_i$. Note that up to homeomorphism, it does not depend on $i$. For $n \geq 2$, there is an embedding $\phi: \frac{1}{2}D^{2n+2}(C) \rightarrow D^{2n}(C)$, that sends pieces to pieces and wedges to wedges, and restricts to an embedding $P_i \rightarrow D^{2n}(C)$ as shown in Figure \ref{fig:inclusion map}.     When $n=1$, we define $\phi$ so that it is not an embedding. Instead, $\phi$ denotes the quotient map that identifies opposite sides of $\Sigma_i$. The restriction $\phi: P_i \rightarrow D^{2}(C)$ still defines a valid embedding.
\end{definition}

\begin{figure}[htbp]
    \centering
    \includegraphics[scale=0.4]{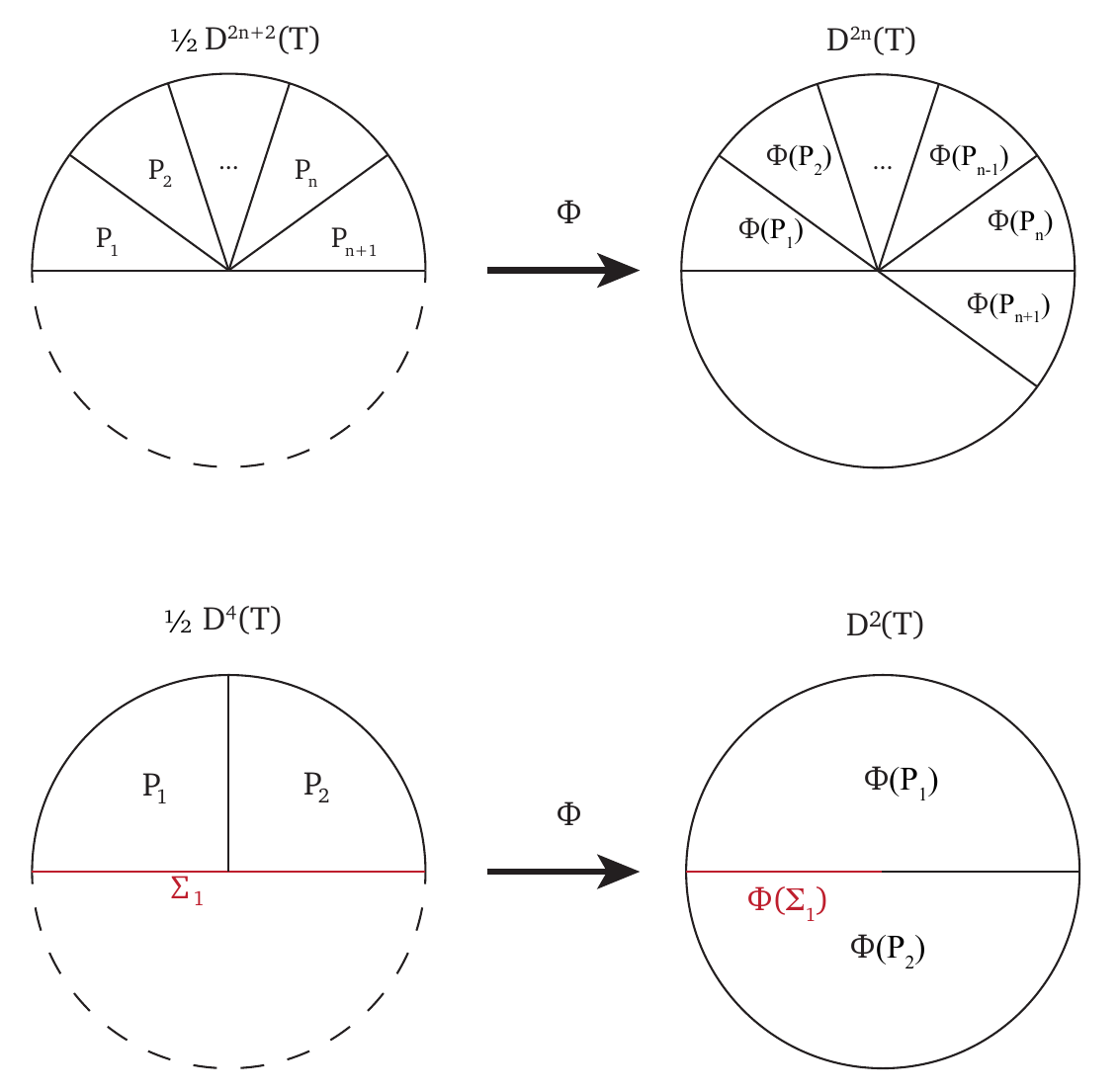}
    \caption{Including pieces into the $2n$ replicant.
    }
    \label{fig:inclusion map}
\end{figure}

\begin{definition} The \textit{central axis} $\mathbf{E} = \bigcap_{i= 1}^{i= n} \Sigma_i$ is the intersection of the sheets, which consists of two closed intervals. Note that any distinct pair of sheets intersects in just $\mathbf{E}$. (In the case $n=1$, we take $E$ to be $\{ b, t\} \times I$, even thought there are no additional surfaces intersecting $\Sigma_1$. 

We define $\partial M$ to be the boundary of $H_{2n,C}$. 
\end{definition}

\begin{definition} \label{pairinc} Given a surface $X$ in a 3-manifold $Y$, we say the pair $(Y,X)$ is incompressible (respectively $\partial$-incompressible) if $X$ is incompressible (resp. $\partial$-incompressible) as a surface in $Y$. If $(Y, X)$ is both incompressible and $\partial$-incompressible, we say that it is an {\it essential pair}.
\end{definition}

\subsection{Outline of the proof of the General Theorem}

We derive Theorem \ref{cor:2_replicant}  by proving the following.

\begin{theorem} \label{thrm:general_replicant}
    Let $C$ be a charm such that $(P_i, \widehat{\Sigma}_i)$ forms an essential pair. If $D^{2n}(C)$ is hyperbolic then $D^{2m}(C)$ is hyperbolic  for all $m \geq n \geq 1$.
    \end{theorem}

    We then prove that in the case of Theorem \ref{cor:2_replicant}, $(P_i, \widehat{\Sigma}_i)$ is an essential pair. However, we note here that $D^{2n}(C)$ can be  hyperbolic even when  $(P_i, \widehat{\Sigma}_i)$ is not an essential pair. For instance, in Figure \ref{charmdisk}, we see an alternating charm $C$ such that $\widehat{\Sigma}_i$ is not incompressible. The boundary of a compressing disk is shown in red. This charm is neither 2-hyperbolic nor 4-hyperbolic since the disk reflects to generate an essential sphere in $D^2(C)$ and an essential torus in $D^4(C)$. However, $C$ is 6-hyperbolic. This particular charm is an example of an alternating charm that fails the conditions we provide in \cite{AdamsKang} for an alternating charm to be 2-hyperbolic.

\begin{figure}[htbp]
    \centering
    \includegraphics[scale=0.6]{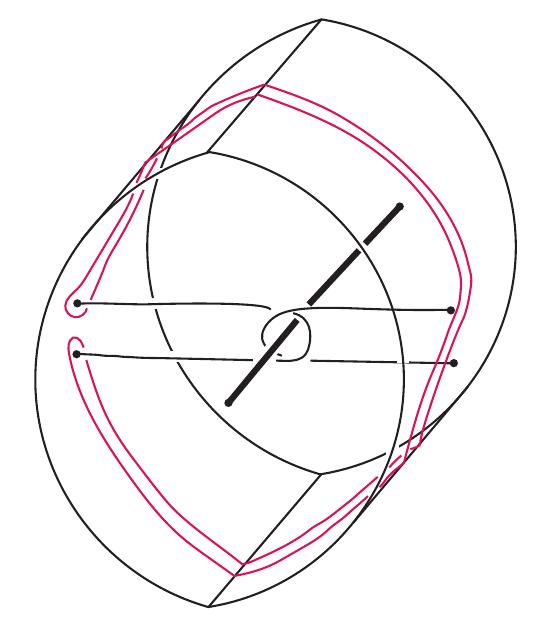}
    \caption{An alternating charm with compressible $\widehat{\Sigma}_i$, such that $D^2(C)$ and $D^4(C)$ are not hyperbolic but $D^6(C)$ is.}
    \label{charmdisk}
\end{figure}

    Moreover, Theorem \ref{thrm:general_replicant} is interesting in its own right. We include an example of a charm in Figure \ref{4hypincomp} where the shaded disk becomes an essential annulus in $D^2(C)$, so $D^2(C)$ is not hyperbolic. However, $D^4(C)$ is hyperbolic, as can be checked with SnapPy \cite{SnapPy},  and it is not hard to show that $(P_i, \widehat{\Sigma}_i)$ is essential. So,  Theorem \ref{thrm:general_replicant} implies $D^{2m}(C)$ is hyperbolic for all $m \geq 2$.

    \begin{figure}[htbp]
    \centering
    \includegraphics[scale=0.4]{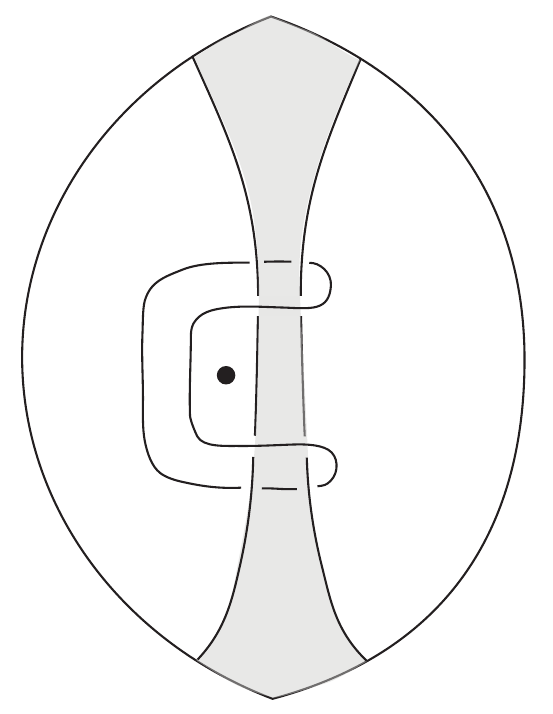}
    \caption{A charm with essential $\widehat{\Sigma}_i$, such that $D^2(C)$ is not hyperbolic but $D^4(C)$ is, and therefore  all $D^{2m}(C)$ are for $m \geq 2$.}
    \label{4hypincomp}
\end{figure}

    We provide an inductive proof of Theorem \ref{thrm:general_replicant}. By the work of Thurston, it is enough to show that $D^{2n+2}(C)$ contains no essential spheres, disks, tori or annuli. In Subsection \ref{subsect:sheetsAndEssentiality}, we show the sheets in $D^{2n+2}(C)$ are essential. In Subsection \ref{subsect:eliminatingSpheresAndDisks}, we use essentiality of the sheets to isotope any essential spheres or disks to lie in a single piece. We then place them in $D^{2n}(C)$, contradicting its hyperbolicity. In Subsection \ref{subsect:eliminatingEssentialTori}, we address the torus case via a slight modification of arguments provided in Theorem 5.1 of \cite{adams2021lower}. In Subsection \ref{subsect:eliminatingEssentialAnnuli}, we use the intersection of any annulus with $\mathbf{S'}$ to control its behavior. After doing so, we tackle the remaining cases by cutting and gluing the annuli to rule out their essentiality.

\medskip

Throughout the rest of this section, we assume $D^{2n}(C)$ is hyperbolic and we assume $(P_i, \widehat{\Sigma}_i)$ is an essential pair. We let $L = L_{2n,C}$ be the link in $D^{2n}(C)$ and $L' = L _{2n+2,C}$ the link in $D^{2n+2}(C)$.

\subsection{Sheets and Essentiality}\label{subsect:sheetsAndEssentiality}

\begin{lemma} \label{Pirred}
    The piece $P_i$ is irreducible.
\end{lemma}

\begin{proof}
    We have $P_i \subseteq D^{2n}(C)$ which contains no essential spheres.  So, a sphere in $P_i$ that does not bound a ball in $P_i$ must bound a ball to the outside in $D^{2n}(C)$. But $D^{2n}(C)$ contains stakes and hence has a higher genus boundary to the outside of the sphere, meaning the outside of the sphere cannot be a ball, a contradiction.
\end{proof}

\begin{lemma} \label{boundarypieceinc} Given a piece $P_i$, $\partial P_i \setminus \widehat{\Sigma}_i$ is incompressible in $P_i$.
\end{lemma}

\begin{proof} Suppose not. So, there is a disk $D$ in $P_i$ such that $\partial D \subset \partial P_i \setminus \widehat{\Sigma}_i$ and $\partial D$ does not bound a disk in $\partial P_i \setminus \widehat{\Sigma}_i$. But, we can consider $P_i$ as a subset of $D^{2n}(C)$, and because $\partial D$ does not bound a disk in $\partial P_i \setminus \widehat{\Sigma}_i$, it also cannot bound a disk in $\partial M$ since all other pieces that make up $D^{2n}(C)$ have stakes giving additional genus to $\partial M$.  However, $D^{2n}(C)$ is hyperbolic and hence $\partial$-irreducible, a contradiction to the existence of $D$.

\end{proof}

\begin{lemma}\label{lemma:pieceNoEssentialDisks}
    The replicant $D^{2n+2}(C)$ has no essential disks lying entirely in some $P_i$.
\end{lemma}

\begin{proof}
   Suppose there were such a disk $D$.  Then, considering  $P_i$ as a subset of $D^{2n}(C)$ implies $D$ is properly embedded in $D^{2n}(C)$, with boundary that does not bound a disk in $\partial P_i \setminus \widehat{\Sigma}_i$, and does not bound a disk in $\partial D^{2n}(C) \setminus \partial P_i$ since that portion of $\partial D^{2n}(C)$ corresponds to stakes and therefore has higher genus. Therefore, the disk is essential in $D^{2n}(C)$, a contradiction to its hyperbolicity.
\end{proof}

\begin{definition} \label{minimalsurf} To each compact essential surface $S$ properly embedded in $D^{2n+2}(C)$ we can isotope so that $S$ and $\mathbf{S'}$ are transverse and $S$ intersects $E$ in a finite set of points that do not include $E$'s endpoints. Then, we associate the pair $(|E \cap S|,|\pi_0(\mathbf{S'} \cap S)|)$.

We say $S$ has {\it minimal complexity in its isotopy class} if it is minimal with respect to this pair under lexicographic ordering under isotopy. We say an essential properly embedded surface $S$ of a given topological type (disk, sphere, torus or annulus) has {\it minimal complexity in its type} if it is minimal with respect to this pair over all essential surfaces of this type properly embedded in the manifold. In some cases, we will drop the word ``complexity'' when the meaning is clear. 
\end{definition}

\begin{lemma}  \label{lemma:splitting_incomp}
    The pairs $(D^{2n}(C), \Sigma_i)$ and $(D^{2n+2}(C), \Sigma_i)$  are incompressible.
   \end{lemma}

\begin{proof} Since $D^{2n}(C)$ is hyperbolic and $\Sigma_i$ is the fixed point set of a reflection, it must be totally geodesic, which implies it is incompressible to either side. 
    We now show $(D^{2n+2}(C),\Sigma_i)$ is an incompressible pair. 
    Choose a compression disk $D$  for $\Sigma_i$ that is minimal complexity in its type. 
    
    Suppose $D$ does not intersect $E$. If $D$ does not intersect other sheets, it is a compression disk for $\widehat{\Sigma}_i$ in $P_i$, contradicting that $(P_i, \widehat{\Sigma}_i)$ is an essential pair. 
    
    Thus, we assume $D$ does intersect other sheets. Because $D\cap E =\emptyset$, every intersection curve is a simple closed curve  on $D$. Pick an innermost curve $\gamma$, bounding a disk $D' \subset D$. Note $\gamma$ lies on some wedge $\widehat{\Sigma}_j$ and $D'$ lies in the piece $P_j$. Since wedge-piece pairs are essential, there exists a disk $D'' \subset \widehat{\Sigma}_j$ such that $\partial D'' = \partial D'$. Together $D' \cup D''$ forms a sphere which bounds a ball in $P_j$ by Lemma \ref{Pirred}. One may isotope $D$ to $D''$ and past the wedge through the ball, eliminating an intersection curve and contradicting minimality.

    Thus, assume $D \cap E \neq \emptyset$. Consider the set of intersection curves $\mathbf{S'} \cap D$ on $D$. Recall $E \subset \Sigma_i$ and $D \cap \Sigma_i = \partial D$, so $D$ intersects $E$ only on its boundary and no two intersection curves intersect on the interior of $D$. By the previous argument, simple closed intersection curves contradict the minimality of $D$. Thus, all intersection curves are disjoint properly embedded arcs. Fix an outermost arc $\gamma$ on $D$. Together with an arc $\alpha \subset \partial D$, $\gamma$ co-bounds a disk $D' \subset D$ containing no intersection curves. Note $\gamma \cup \alpha$ is a simple closed curve on $\widehat{\Sigma}_i$ and $D' \subset P_i$. 
    Since $(P_i, \widehat{\Sigma}_i)$ is an essential pair, $\gamma \cup \alpha$ bounds some disk $D'' \subseteq \widehat{\Sigma}_i$. Then $D' \cup D''$ bounds a ball in $P_i$ and we can use it to isotope $\alpha$ along $\Sigma_i$ and through $E$ to eliminate $\gamma$, a contradiction to $D$ being minimal complexity.
    
    See Figure \ref{fig:sheet incompressible}.
\end{proof}

\begin{figure}[htbp]
    \centering
    \includegraphics[scale=0.4]{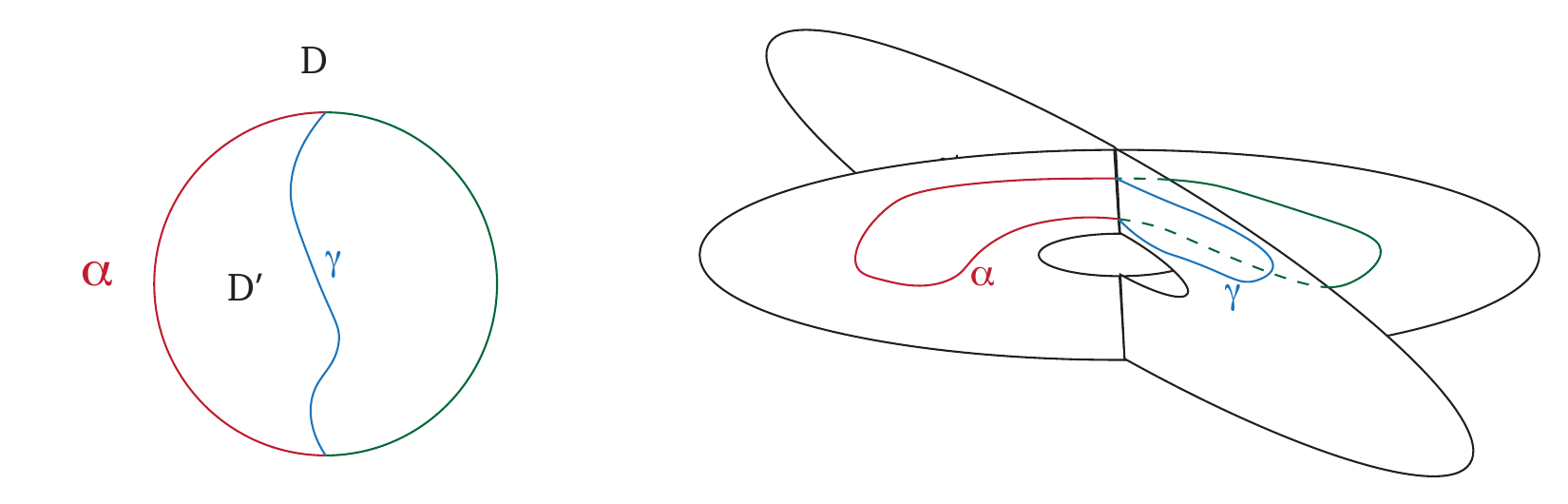}
    \caption{An outermost intersection arc on $D$ yields a disk in a piece $P_i$.}
    \label{fig:sheet incompressible}
\end{figure}

\begin{lemma}
    The pairs $(D^{2n+2}(C), \Sigma_i)$ and $(D^{2n}(C), \Sigma_i)$ are $\partial$-incompressible pairs.
    \label{lemma:splitting_bdy_incomp}
\end{lemma}
\begin{proof}Again, since $D^{2n}(C)$ is hyperbolic and $\Sigma_i$ is the fixed point set of a reflection, it must be totally geodesic, which implies it is $\partial$-incompressible to either side.

    We now show $(D^{2n+2}(C), \Sigma_i)$ is $\partial$-incompressible. Choose a minimal complexity $\partial$-compression disk $D$. Recall $\partial D$ is composed of an arc $\alpha \subseteq \Sigma_i$ and an arc $\beta \subseteq \partial D^{2n+2}(C)$. Since $E \subset \Sigma_i$ and $D \cap \Sigma_i \subset \partial D$, then $D \cap E \subset \partial D$. Thus no two intersection curves of $\mathbf{S'} \cap D$ meet in the interior of $D$. As argued in Lemma \ref{lemma:splitting_incomp}, simple closed curves of intersection contradict minimality. Thus, we need only concern ourselves with arcs of intersection, the boundary points of which lie in $\partial D$. Note an outermost 
    intersection arc $\gamma$ with both endpoints on $\alpha$, cuts off a disk $D' \subset D$, with $\partial D' \subset \widehat{\Sigma}_j, D' \subset P_j$. A similar argument as in the last paragraph of the proof of Lemma \ref{lemma:splitting_incomp} shows such intersections can be isotoped away.

    We now consider intersection arcs such that not both their endpoints are in $\alpha$.
    First, suppose $D$ does not intersect $E$. Choose an outermost arc of intersection $\gamma$, which cuts off a disk $D_1 \subset D$, containing no intersection curves.
    
     If one endpoint of $\gamma$ is on $\alpha$, since $\gamma\cap E=\emptyset$, for the other endpoint of $\gamma$ to be on $\beta$ it must be in $\partial \Sigma_i\cap D\subset\alpha$, which reduces
    to the previous case. Thus, we assume both endpoints of $\gamma$ lie in $\beta$, and $\partial D_1 = \gamma \cup \beta'$, where $\gamma \subseteq \widehat{\Sigma}_j$ for some $j$ and $\beta' \subseteq \partial D^{2n+2}(C)$. Since $(P_j, \widehat\Sigma_j)$ is an essential pair, $\gamma$ must cobound a disk $D_2 \subset \widehat{\Sigma}_j$ with an arc on $ \partial \widehat{\Sigma}_j$. Glue $D_1$ to $D_2$ along $\gamma$ forming a disk $D_1 \cup D_2$ properly embedded in $D^{2n+2}(C)$. By Lemma \ref{lemma:pieceNoEssentialDisks}, there exists a disk $D_3 \subset \partial D^{2n+2}(C)$ with $\partial D_3 = \partial (D_1 \cup D_2)$. Together $D_3 \cup D_1 \cup D_2$ form a sphere which must bound a ball by Lemma \ref{Pirred}. One may isotope $D_1$ to $D_2$ and then past $\widehat{\Sigma}_j$, while keeping $\beta'$ on $D_3 \subset \partial D^{2n+2}(C)$. This reduces the number of intersection curves, contradicting minimality of $D_1$. 
    
    Suppose $D$ intersects $E$. Note $E \cap D$ cuts $\alpha$ into a series of subarcs. Take a subarc $\alpha' \subset \alpha$ with one endpoint shared with an endpoint of $\beta$. 
    
    Consider the intersection curves $\mathbf{S'} \cap D$. There is an arc $\gamma$ such that $\gamma$ shares one endpoint with $\alpha'$ and has its other endpoint in $\beta$.  Note $\gamma$ cuts off a sub-arc $\beta' \subset \beta$, and an  outermost disk $D' \subset D$, such that $\partial D' = \alpha' \cup \gamma \cup \beta'$. Note $\alpha' \cup \gamma \subset \widehat{\Sigma}_j$, $\beta' \subset P_j \cap \partial D^{2n+2}(C)$ and $D' \subset P_j$. Since $(P_j, \widehat{\Sigma}_j)$ is an essential pair, $\alpha' \cup \gamma$ cobounds a disk $D''$ in $\widehat \Sigma_j$ with an arc on $\partial \widehat{\Sigma}_j$. Since $\alpha'$ intersects $E$, then $D''$ intersects $E$. By Lemma \ref{boundarypieceinc}, $\partial (D'\cup D'')$ bounds a disk $D'''$ in $\partial P_i \setminus \widehat{\Sigma}_i$. By Lemma \ref{Pirred}, the sphere $\D \cup D' \cup D'''$ bounds a ball in $P_i$. We can use this ball to isotope $D'$ to $D''$ and then beyond, eliminating an intersection of $D$ with $E$ and contradicting its minimality.
    See Figure \ref{fig:sheetsBoundaryIncompressible}.
\end{proof}

\begin{figure}[htbp]
    \centering
    \includegraphics[width=0.5\linewidth]{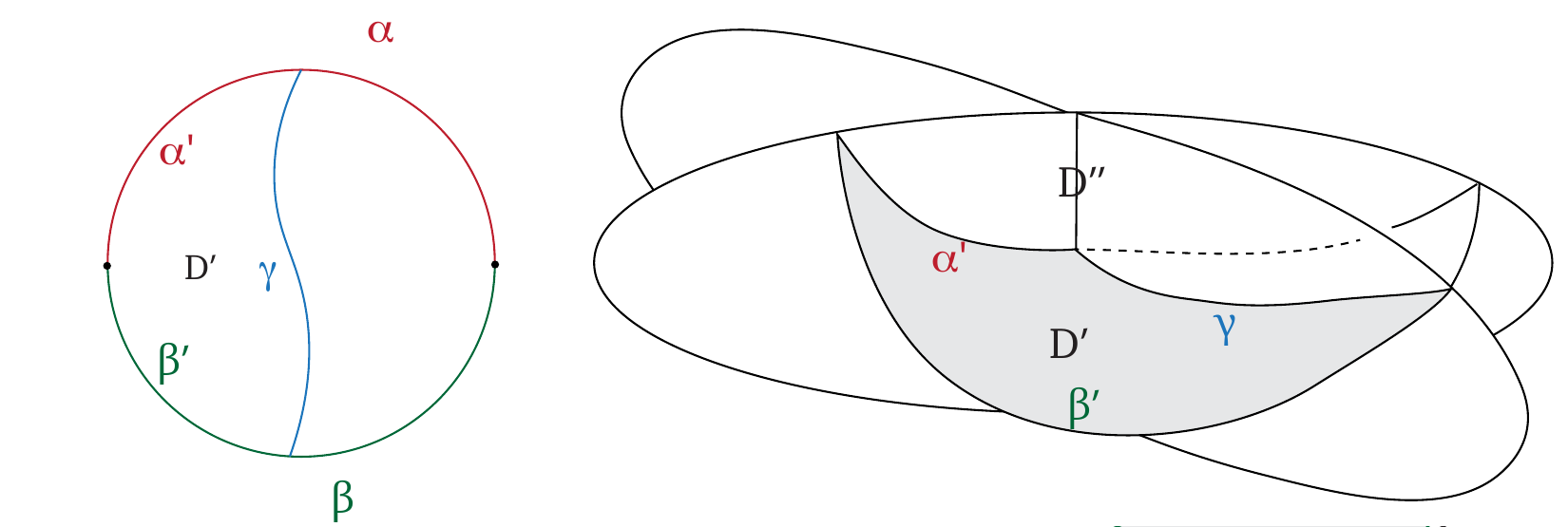}
    \caption{Boundary compressions either induce boundary compressions with a wedge or can be isotoped to avoid $E$.}
    \label{fig:sheetsBoundaryIncompressible}
\end{figure}

\begin{definition}Let $S$ be a properly embedded surface in $D^{2n+2}(C)$. Consider the intersection graph $G$, where vertices of $G$ are points on $E \cap S$, and the edges are arcs on $\mathbf{S'} \cap S$. A region $R$ in the complement of $G$ on $S$ is an \textit{interior polygonal region} if the closure of $R$ is an embedded disk on $S$ and $\partial R \cap \partial S = \emptyset$. Analogously $R$ is an $\textit{exterior polygonal region}$ if the closure of $R$ is an embedded disk and $\partial R$ meets $\partial S$ in a single arc. See Figure \ref{fig:interiorPolygonal}. Note that in either the case of an interior or exterior polygonal region, there may or may not be vertices coming from intersections of $S$ with $E$.
\end{definition}

\begin{lemma}[Graph Lemma] \label{lemma:GraphLemma} If $S$ is a minimal complexity properly embedded essential surface in its  isotopy class  in $D^{2n+2}(C)$, $G$ generates no interior or exterior polygonal regions on $S$.  

\end{lemma}

\begin{figure}[htbp]
    \centering
    \includegraphics[scale=0.15]{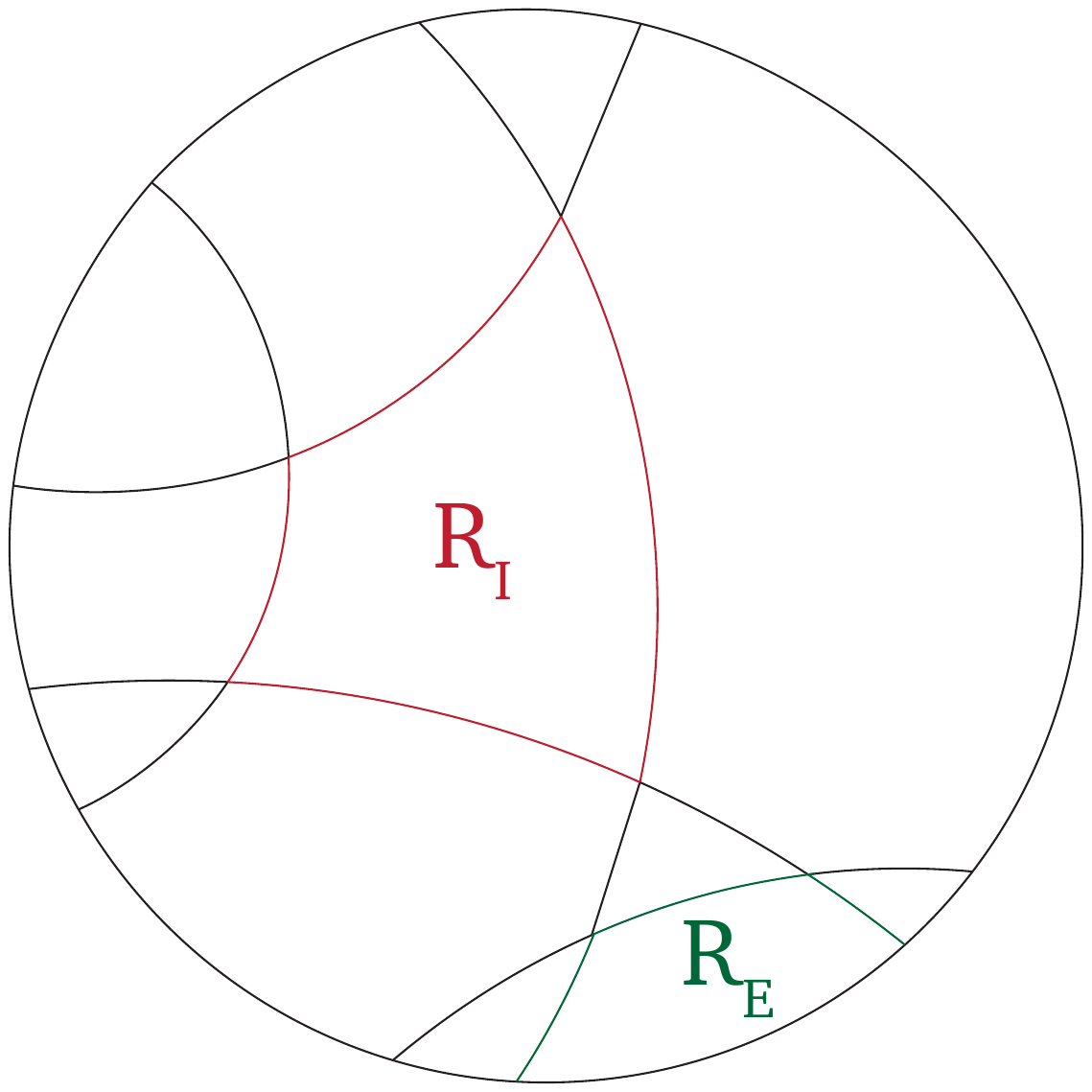}
    \caption{An interior polygonal region $R_I$ and an exterior polygonal region $R_E$ on a disk.}
    \label{fig:interiorPolygonal}
\end{figure}

\begin{proof}
Note any interior polygonal region $R_I$ has $\partial R_I \subset \widehat{\Sigma}_j$ and $R_I \subset P_j$. Because $(P_j,\widehat{\Sigma}_j)$ is an essential pair, $\partial R_I$ must bound a disk in $\widehat{\Sigma}_j$ and because $P_j$ is irreducible, the sphere given by these two disks must bound a ball in $P_j$. Hence, we can  isotope $R_I$ past $\widehat{\Sigma}_j$, decreasing complexity of $S$, a contradiction.

In the case there is an exterior polygonal region $R_E \subset P_j$, $\partial R_E$ intersects $\widehat{\Sigma}_j$ in one arc  $\alpha$ that can cross $E$ multiple times and otherwise intersects $\partial P_j \setminus \widehat\Sigma_j$ in a single arc $\beta$. 
Because $(P_j,\widehat{\Sigma}_j)$ is essential, $\widehat{\Sigma}_j$ is $\partial$-incompressible in $P_j$ and hence $\alpha$ cuts a disk $R'$ from $\widehat\Sigma_j$. 

Then, after a slight isotopy,  $R_E \cup R'$ is a disk in $P_i$ with boundary in $\partial P_i \setminus \widehat{\Sigma}_i$. By Lemma \ref{boundarypieceinc}, $\partial (R_E \cup R')$ must bound a disk $R''$ in $\partial P_i \setminus \widehat{\Sigma}_i$. Hence, $R_E \cup R' \cup R''$ is a sphere in $P_i$ which must bound a ball in $P_i$ by Lemma \ref{Pirred}. We can use that ball to isotope $R_E$ to $R'$ and a little further to lower the  complexity of $S$, a contradiction.
\end{proof}

\subsection{Eliminating Essential Spheres and Disks}\label{subsect:eliminatingSpheresAndDisks}

\begin{lemma}\label{lemma:EssentialSphere}
     The replicant $D^{2n+2}(C)$ contains no essential spheres.    
    \end{lemma}
\begin{proof}
Let $S$ be a minimal essential sphere in $D^{2n+2}(C)$. Since $S$ is a sphere, any innermost intersection curve in $G= \mathbf{S'} \cap S$ bounds a disk, yielding an interior polygonal region. Thus by Lemma \ref{lemma:GraphLemma}, $\mathbf{S'} \cap S = \emptyset$. 
Therefore $S$ lies in some piece. By Lemma \ref{Pirred}, $S$ cannot be essential.
\end{proof}

\begin{lemma}\label{lemma:EssentialDisks}
    The replicant $D^{2n+2}(C)$ contains no essential disks.
      \end{lemma}
\begin{proof}

Let $D$ be a minimal essential disk. If $D \cap \mathbf{S'}$ is nonempty, because any graph on a disk has either an interior or exterior polygonal region, $D$ cannot be minimal. Thus a minimal disk cannot intersect $\mathbf{S'}$, in which case it is not essential by Lemma \ref{lemma:pieceNoEssentialDisks}. 
\end{proof}

\subsection{Eliminating Essential Tori}\label{subsect:eliminatingEssentialTori}

\begin{lemma} \label{lemma:non_parallel_torus}
Let $A$ be an annulus in $P_i$ with boundary circles $\partial A_0$, $\partial A_1$ on $\widehat{\Sigma}_i$ such that $A \cap E = \emptyset$. Let $T$ be the torus formed in $D^{2n}(C)$ by gluing $A$ to its reflections, as in Figure \ref{fig:reflectingAnnuliToTori}. Then, $T$ is not parallel to $\partial M$.
\end{lemma}

\begin{proof}
When $P_i$ contains more than one stake or when $n > 1$ the genus of $\partial M$ is at least 2 and the result follows immediately. Thus, without loss of generality, we assume each piece contains only one stake and $n = 1$.  If $T$ is parallel to $\partial M$, the longitude of $T$ must link with each of the two stakes in $D^{2}(C)$. Hence, each of the boundary components $\partial A_0$ and $\partial A_1$ must be concentric on $\widehat{\Sigma}_i$ and wind once around the stake in $P_i$. But this forces them to intersect $E$, a contradiction. 
\end{proof}

\begin{figure}[htbp]
    \centering
    \includegraphics[scale=0.6]{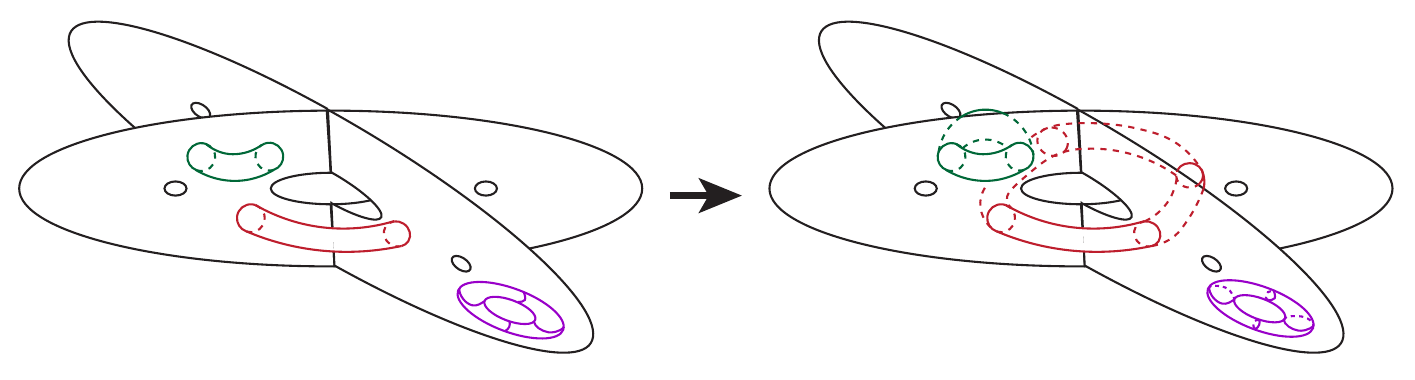}
    \caption{The cases of annuli in $P_i$  that don't intersect $E$ and the corresponding tori obtained by reflection.}
    \label{fig:reflectingAnnuliToTori}
\end{figure}

\begin{lemma}
    Let $T \subset D^{2n+2}(C)$ be a torus such that $T\cap \mathbf{S'} = \emptyset$. Then $T$ is not essential.
   \label{lemma:torus_in_piece}
\end{lemma}

\begin{proof}
The torus $T$ lies in some piece $P_i$, and therefore $\phi(T) \subseteq D^{2n}(T)$ is compressible or $\partial$-parallel. Let $D$ be a compression disk for $\phi(T)$. 

Consider the intersection graph $G = \mathbf{S'} \cap D$. Pick $D$ minimal among all such compression disks. Note by the argument presented in the proof of Lemma \ref{lemma:GraphLemma}, we may isotope away any interior polygonal regions on $D$ without altering $\partial D$. 
Thus, since $D$ is minimal, $G$ must contain no interior polygonal regions. 

Note $\partial D \subset T$, so $\partial D \cap \mathbf{S'} = \emptyset$, thus $G$ contains no properly embedded arcs of intersection. Thus if $G \neq \emptyset$, $G$ must contain an interior polygonal region, contradicting the minimality of $D$. Thus $D$ must lie in $P_i$, and the pre-image $\phi^{-1}(D)$ forms a compression disk for $T$ in $D^{2n+2}(C)$.
 
Suppose $\phi(T)$ is incompressible and  $\partial$-parallel.
If $\phi(T)$ is parallel to $\partial M$, then $\phi(T)$ is one boundary component of a $T \times I$ with $\partial M$ as the other boundary component. But since $\phi(T)$ is contained in $P_i$ and the other pieces contain some part or all of various link components, this cannot be the case. 
 
 Then, without loss of generality, assume $\phi(T)$  is parallel to a link component $K$. Thus, there is an isotopy taking $\phi(T)$ to $K$, by retracting meridian circles until they form meridians of $K$.  Note this isotopy takes place entirely inside $P_i$. Then the image of this isotopy under $\phi^{-1}$, yields an isotopy taking  $T$ to $\phi^{-1}(K)$ in $D^{2n+2}(C)$, a contradiction.
\end{proof}

We follow the arguments in Lemma 5.1 of \cite{adams2021lower} to rule out essential tori.

\begin{lemma}\label{primelink}
    The link $L'$ in $D^{2n+2}(C)$ is prime, by which we mean $D^{2n+2}(C)$ has no essential annulus realized as an essential sphere  punctured twice by $L'$.
\end{lemma}
\begin{proof}
    Suppose such a twice-punctured sphere $S$ exists. We choose $S$ to be minimal complexity. 
If $S$ intersects one of the surfaces $\Sigma_i$ in a closed curve that is trivial on the punctured sphere, we can isotope to remove the intersection.  If $S \cap E = \emptyset$, then if there is a nontrivial intersection curve on $S$ with one $\Sigma_i$, choose an innermost such, so it bounds a punctured disk $D$ on $S$ containing no other intersection curves. Then, $D$ is properly embedded in a $P_i$ with boundary in $\widehat\Sigma_i$. We can reflect $D$ in $D^{2n}(C)$ to obtain a twice-punctured sphere in $D^{2n}(C)$, which by hyperbolicity, must bound a ball containing an unknotted arc of $L$. However, then $D$ cuts a ball containing an unknotted arc from $P_i$ in $D^{2n+2}(C)$, and we can isotope $D$ to $\Sigma_i$ and through to lower the complexity of $S$, a contradiction.

Now consider the case that $S \cap E \neq \emptyset$. Let $G$ be the $2n+2$-valent graph that results on $S$.

Suppose first that $G$ consists of a single connected component. 
 Then, for the corresponding unpunctured sphere, Euler characteristic $v - e + f = 2$ yields $f = nv + 2$. Since $n, v \geq 1$, there are at least three faces, only two of which can be punctured. Hence there is an interior polygonal region, contradicting the Graph Lemma.

    Now, suppose there is more than one component of $G$. 
    Choose one component. It contains at least three faces, and if any one of them does not contain a puncture or another component, it contradicts the Graph Lemma. So, at least one of these faces contains another component and possibly all three do. Take an innermost component in each face, or take the face itself.  Each such component contains  at least two faces (not counting the outer one), so we have at least four such faces, and only two punctures to place in them. So, there will be a face that contradicts the Graph Lemma. 
     \end{proof}

\begin{lemma} \label{lemma:EssentialTorus}
    $D^{2n+2}(C)$ contains no essential tori. 
\end{lemma}
\begin{proof}
    Let $T$ be an essential torus in $D^{2n+2}(C)$. Choose $T$ to be minimal complexity in its type.

    First suppose $T$ does not intersect $E$. Then by incompressibility of each $\Sigma_i$, $T$ intersects $\mathbf{S'}$ in a finite number of parallel curves on $T$ that divide $T$ into a collection of disjoint annuli, each lying in a single piece. (See Figure \ref{fig:reflectingAnnuliToTori} again.)
    
    If some annulus $A$ in a piece $P_i$ has both boundary circles on the same sheet $\Sigma_i$, it replicates in the $2n$-replicant $D^{2n}(C)$ to a collection of $n$ tori, none of which can be essential. Let $T'$ be one of these tori. If $T'$ is boundary parallel, then it is boundary parallel to a link component by Lemma \ref{lemma:non_parallel_torus}, so $A$ is also parallel to a link component. This parallellism lets us meridianally compress $T$ in the $2n$-replicant, either yielding an essential annulus and contradicting primeness of the link from Lemma \ref{primelink} or showing $T$ was boundary parallel to begin with, a contradiction. 
    
    If $T'$ is compressible, it cannot compress along a curve parallel to the boundary of $A$, as then $A$ and hence $T$ would be compressible. Thus,  there is a compression disk with longitude boundary on $T'$ if the boundaries of $A$ are not concentric and meridian boundary if they are concentric. 
    By incompressibility of $\Sigma_i$, we may isotope the compression disk so it intersects $\Sigma_i$ in a single arc, and its boundary subdivides into two arcs, one on $A$ and one on the reflection of $A$. Let $D'$ be the sub-disk bounded by the boundary arc in $A$ and the arc in $\Sigma_i$. 

Let $D''$ be the disk boundary of $N(A \cup D')$. That is, $D''$ is obtained by gluing  the disk obtained by cutting $A$ open along $\partial D \cap A$ to two copies of $D'$ along $\partial D \cap A$. 
Then, $\partial D'' \subset \Sigma_i$, so it bounds a disk $D'''$ in $\Sigma_i$. Then $D''\cup D'''$ is a sphere in $P_i$ which must bound a ball containing $A$, implying $A$ is either compressible in the ball, or can be isotoped through the ball to lower the complexity of $T$, in either case a contradiction.

    Thus, every annulus in each piece has its boundary circles on distinct sheets in the starburst. Following the arguments in Section 5 of \cite{adams2021lower}, we may use essentiality of the torus and the sheets in our starburst to isotope $T$ so that it is composed of reflected copies of the same annulus $A$. But then these reflected copies give us a torus $T''$ in the $2n$-replicant, which is either boundary parallel or compressible. Boundary parallelisms of $T''$ must be to the link by Lemma $\ref{lemma:non_parallel_torus}$, which force $A$ to be boundary parallel to an arc in a link component. Therefore, $T$ is boundary parallel in $D^{2n+2}(C)$, a contradiction. 
    
If $T''$ is instead compressible, the compression must be to the outside, so we may choose it to be a longitude. By incompressibility and choosing a minimal such compression disk, we may isotope the disk such that it is composed of reflected copies of a disk in $C$, with boundary that is made up of one arc a nontrivial arc on $A$ and one arc on a wedge $\widehat{\Sigma}_i$.
We may reflect the disk back and forth in $D^{2n+2}(C)$ to obtain a compression disk for $T$ in $D^{2n+2}(C)$, a contradiction to its being essential here.

    Finally, suppose $T$ intersects $E$. Consider the intersection graph $G$ of $T$ with $\mathbf{S'}$. Because sheets in $\mathbf{S'}$ intersect only at $E$ and all sheets intersect there, and because $T$ separates $D^{2n+2}(C)$,  the intersection graph has an even number of vertices, each of valency $2n+2$.

First, suppose $G$ has only one component. Then, there must be a complementary region that is a disk, again yielding an interior polygonal region that is a disk, and contradicting the Graph Lemma. If $G$ has more than one component, take an innermost such. It will still have an interior polygonal region and thus generate a contradiction.
\end{proof}

\subsection{Eliminating Essential Annuli}\label{subsect:eliminatingEssentialAnnuli}

Let $A$ be an annulus and $\partial A_0, \partial A_1$ denote its boundary circles. We classify annuli in $D^{2n+2}(C)$ into three cases. Let $\partial N(L')$ denote the union of the link boundary components.
\begin{enumerate}
\item $A$ is a \textit{Type I} annulus if $\partial A_0 \subset \partial M$ and $\partial A_1 \subset \partial M$
\item $A$ is a \textit{Type II} annulus if $\partial A_0 \subset \partial M$ and $\partial A_1 \subset \partial N(L')$
\item $A$ is a \textit{Type III} annulus if $ \partial A_0 \subset \partial N(L')$ and $\partial A_1 \subset \partial N(L')$
\end{enumerate}

We focus at first on eliminating essential Type I annuli. We do so by first using the Graph Lemma to control the intersection curves with $\mathbf{S'}$. After doing so we are left with two cases which we tackle individually.
We eliminate Type II annuli by reducing them to Type I annuli. Type III annuli arise only in Seifert-fibered spaces. We eliminate them by arguing $D^{2n+2}(C)$ is never Seifert fibered.

\begin{remark}\label{annuliinmanifolds}
It is a well-known fact that an annulus in an irreducible, $\partial$-irreducible 3-manifold is $\partial$-parallel if and only if it is $\partial$-compressible. (See for example \cite{hatcher3manifold}.) We use this fact repeatedly without further mention.
\label{rmrk:boundary_interchangability}
\end{remark}

\begin{lemma} \label{annulusmissingsheet}
Let $A$ be a Type I annulus in $D^{2n+2}(C)$ such that $A$ does not intersect a particular sheet $\Sigma_i$. Then $A$ cannot be essential.
\label{lemma:splitting_annulus}
\end{lemma}

\begin{proof}
Suppose $A$ is essential. Since $A$ does not intersect $\Sigma_i$, it lies in half of $D^{2n+2}(C)$ and hence, we can consider $\phi(A) \subset D^{2n}(C)$.  When $n \neq 1 $, $\phi(\Sigma_i)$ splits $D^{2n}(C)$ into two disjoint components. Note $\phi(A)$ lies in one such component. 

Suppose $\phi(A)$ compresses in $D^n(C)$. Let $D$ be the compression disk. 
Taking the union of $D$ with each of the sub-annuli that $\partial D$ cuts $\phi(A)$ into yields two properly embedded disks  in $D^n(C)$. By $\partial$-irreducibility and irreducibility of $D^n(C)$, they each cut a ball from $D^n(C)$. Taking the union of the two balls along $D$ yields a ball $B$ that is cut from $D^n(C)$ by $\phi(A)$. Since $A$ does not intersect $\Sigma_i$ in $D^{2n+2}(C)$, $\phi(A)$ does not intersect $\phi(\Sigma_i)$ in $D^n(C)$. So, either $\phi(\Sigma_i) \cap B = \emptyset$ or $\phi(\Sigma_i) \subset B$. In the first case, The pre-image of $D$ in $D^{2n+2}(C)$  yields a compression disk for $A$, a contradiction. In the second case, the existence of $B$ surrounding $\phi(\Sigma_i)$ prevents any strands of $C$ from intersecting $\phi(\Sigma_i)$ and hence, from intersecting $\Sigma_i$ in 
$D^{2n+2}(C)$, contradicting the rules for a charm.

Now, suppose $\phi(A)$ $\partial$-compresses in $D^n(C)$, with $\partial$-compression disk $D$. Then a regular neighborhood of $D \cup \phi(A)$ has a component that is a disk $D'$  with boundary in $\partial H_{n,C}$, with two parallel copies of $D$ as part of $D'$. Since $\D^n(C)$ is $\partial$-irreducible, $\partial D'$ bounds a disk $D''$ on $\partial H_{n,C}$. Then, $D' \cup D''$ is a sphere, which by irreducibility of $D^n(C)$, bounds a ball in $D^n(C)$. Then $\phi(A)$ is either inside the ball or outside the ball. 

If $\phi(A)$ is inside the ball then $\phi(A)$ is  compressible.
If $\phi(A)$ is outside the ball, then the two copies of $D$ that appear in $D'$ and are therefore on the boundary of the ball glue together to yield a solid torus in $D^n(C)$, through which $\phi(A)$ can be isotoped to $\partial H_{n,C}$. 
In particular, this means we can find a $\partial$-compressing disk $D'''$ for $\phi(A)$ that avoids $\phi(\Sigma_n)$. This disk lifts to a $\partial$-compressing disk for $A$ in $D^{2n+2}(C)$, a contradiction to $A$ being essential in $D^{2n+2}(C)$.
\end{proof}

\begin{lemma}
  Let $A$ be a Type I minimal complexity essential annulus in $D^{2n+2}(C)$. 
  Then, the intersection graph $G = \mathbf{S'} \cap A$ takes one of the two forms depicted in Figure \ref{fig:possibleAnnulusIntersections}.
  \label{lemma:annulus_graph_lemma}
\end{lemma}

\begin{proof}

If $\mathbf{S'} \cap A = \emptyset$, then by Lemma \ref{lemma:splitting_annulus}, $A$ cannot be essential. Thus, $G$ must be a non-empty set of curves. There can be no arcs with  both endpoints on the same edge of $A$ since an outermost such cuts a disk off $A$ which contradicts the minimality of $A$ by Lemma \ref{lemma:GraphLemma}. Any interior vertices in $G$ induce an exterior/interior graph region, which also contradicts the minimality of $A$ by Lemma \ref{lemma:GraphLemma}. It then follows $\mathbf{S'} \cap A$ takes on one of the desired forms.
\end{proof}

\begin{figure}[htbp]
    \centering
    \includegraphics[width=0.8\linewidth]{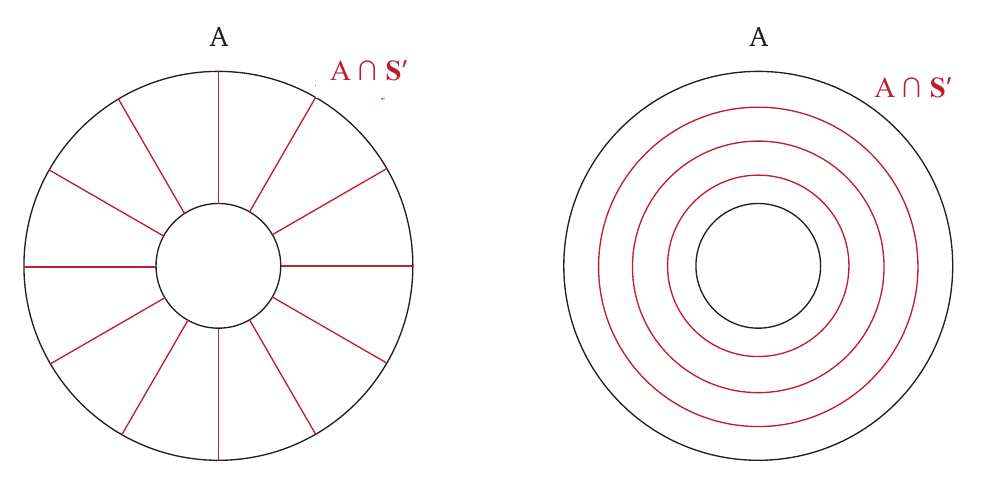}
    \caption{Possible intersections of a minimal Type I annulus with the sheets.}
    \label{fig:possibleAnnulusIntersections}
\end{figure}

\begin{lemma} A  Type I  minimal complexity essential annulus $A$ in $D^{2n+2}(C)$ cannot have $G$ form a series of concentric non-trivial simple closed curves on $A$.
\label{lemma:simpleClosedCurveAnnulusInessential}
\end{lemma}

\begin{proof}
Let $A$ be such an annulus. Note $D^{2n+2}(C)$ contains at least two sheets and by Lemma \ref{lemma:splitting_annulus}, there are at least two concentric circles in $G$. Take an outermost such circle $\gamma_1$ closest to boundary component $\partial A_0$ of $A$, and an adjacent non-boundary circle $\gamma_2$. Note $\gamma_1, \gamma_2$ lie on the same wedge $\widehat{\Sigma}_i$ and bound an annulus $A' \subset A$ that lies in the piece $P_i$ as we saw on the left side of  Figure \ref{fig:reflectingAnnuliToTori}. 

Again consider the torus $T'$ formed  in $D^{2n+2}(C)$ by reflecting $A'$ across the wedges as on the right side of Figure \ref{fig:reflectingAnnuliToTori}. 
Note that $\gamma_1$ and $\gamma_2$ must bound disks, possibly punctured by $L'$, on $\widehat{\Sigma}_i$, since they do not intersect $E$ and $\widehat{\Sigma}_i \setminus E$ is a pair of disks. If such a disk $D$ is unpunctured in the complement of $L'$, choosing the innermost one if the two curves $\gamma_1$ and $\gamma_2$ are concentric, this would give a compression disk for $A$, a contradiction. Hence, $D$ is punctured by $L'$. 

By Lemma \ref{lemma:EssentialTorus}, $T'$ is either compressible or $\partial$-parallel. 
If $T'$ is compressible, then  after compression along compression disk $D$, we obtain a sphere which must bound a ball. Hence, $T'$ bounds a solid torus or a knot exterior. In both cases, $\partial D$ is parallel on $T'$ to the components of $\partial A'$, a contradiction to $A$ being incompressible. Then one boundary component of $N(T' \cup D)$ is a sphere which must bound a ball to the $T'$ side since $\partial M$ is to the other side. Hence $T'$ must bound an unknotted solid torus in the link complement.  In the case $\partial A$ is a longitude of the torus, we can then push $A'$ to $\widehat{\Sigma}_i$ and through to lower the complexity of $A$, a contradiction. In the case $\partial A$ is a meridian of $T'$, then $A$ is compressible, a contradiction.

If $T'$ is $\partial$-parallel, then  by Lemma \ref{lemma:non_parallel_torus}, $T'$ is parallel to a link component $K$. Thus $\gamma_1$ must wrap around a single puncture on $\Sigma_i$, bounding an annulus on $\widehat{\Sigma}_i$ with one boundary component $\gamma_1$ and the other a meridian on $K$. Note $\partial A_0 \cup \gamma_1$ bound an annulus on $A$. Gluing along $\gamma_1$, yields an annulus $A''$ with one boundary circle $\partial A_0$ and one boundary circle a meridian on $K$. 
But, Type $II$ annuli with one boundary a knot meridian and the other on the manifold boundary are automatically essential in both $D^{2n}(C)$ and  $D^{2n+2}$. To see this, they cannot be $\partial$-compressible, and compression discs induce spheres punctured exactly once by a knot meridian, which  cannot occur. Therefore Type $II$ annuli of this kind existing in a piece is a contradiction to hyperbolicity of $D^{2n}(C)$.  
\end{proof}

\begin{lemma}A Type I minimal complexity essential  annulus with $G$ a series of disjoint non-trivial arcs cannot have two adjacent intersection arcs lie on the same $\Sigma_i$.
 \label{lemma:annulus_double_back}
\end{lemma}

\begin{proof} See Figure \ref{fig:annulusDoublingBack}.
Any two adjacent arcs cut a disk $D$ off of $A$. If the two arcs lie on some $\Sigma_i$, reflecting $D$ across $\Sigma_i$, reflects to an annulus $A'$. Since $A'$ intersects only $\Sigma_i$, by Lemma \ref{lemma:splitting_annulus} $A'$ is compressible or $\partial$-compressible. Any $\partial$-compression disk for $A'$ induces a $\partial$-compression disk for $A$, so $A'$ must be compressible.

Since $A'$ is compressible, by Lemma \ref{lemma:EssentialDisks}, $\partial A'_0$ and $\partial A'_1$ bound disks $D_0$ and $D_1$ which lie on $\partial M$. Since $\partial A'_0 \cap \partial A'_1 = \emptyset$, $D_0$ and $D_1$ are disjoint or concentric. If $D_0$ and $D_1$ are disjoint, $A' \cup D_0 \cup D_1$ forms a sphere which bounds a ball. We can isotope $D$ through the ball to eliminate two intersection curves on $A$, contradicting minimal complexity. If $D_0 \subset D_1$ then the boundary of a regular neighborhood of $A \cup D_1$ forms a sphere which bounds a ball, and again, we can isotope $D$ through the ball to eliminate two intersection curves on $A$, contradicting minimality.
\end{proof}

\begin{figure}[htbp]
    \centering
    \includegraphics[trim = {3cm 6cm 3cm 6cm}, scale=0.3]{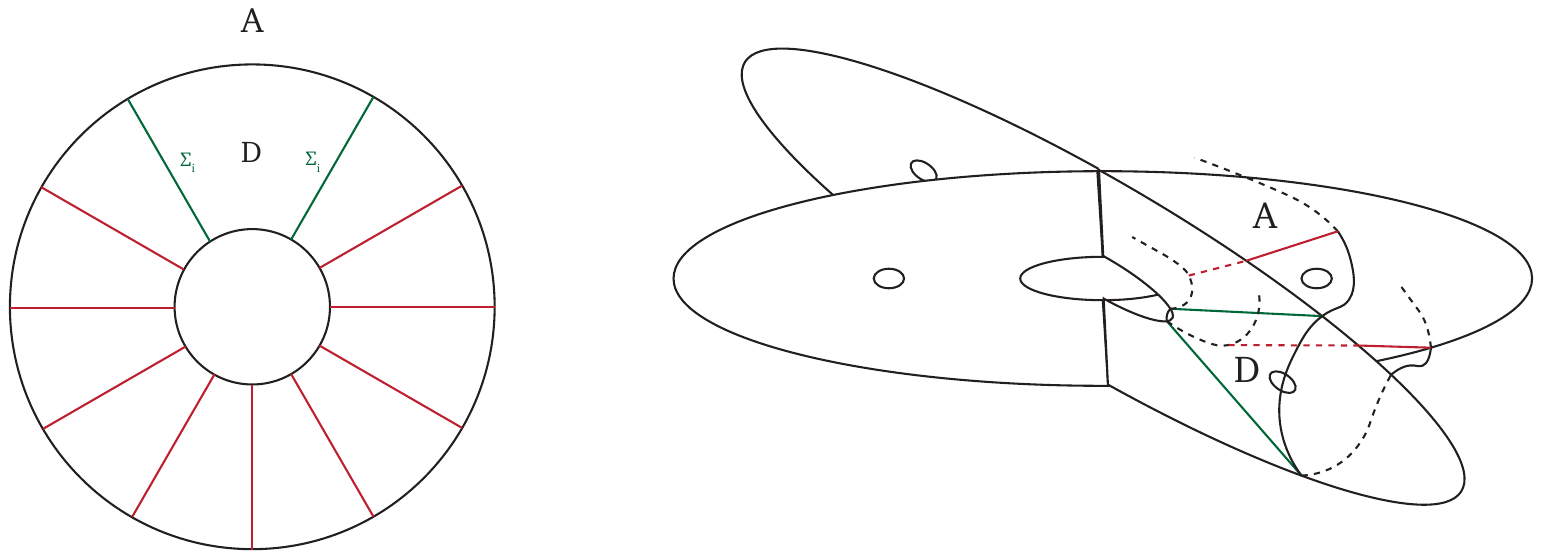}
    \caption{$D$ forms half of an annulus.}
    \label{fig:annulusDoublingBack}
\end{figure}

\begin{lemma} \label{lemma:arcAnnulusInessential}
There are no  Type I minimal complexity essential annuli with $G$ a series of disjoint non-trivial arcs.
 \end{lemma}

\begin{remark} We prove this lemma using a series of other lemmas. Note that by Lemma \ref{lemma:annulus_double_back}, such an annulus intersects the sheets in cyclic order. We wish to show that $A$ is isotopic to an annulus of the form $D \cup D^R \cup ... \cup D \cup D^R$, where each $D$ lies in a distinct piece. Note that the arcs of intersection cut disks off $A$. Consider a pair of adjacent disks that are, without loss of generality,  $D_1\subseteq P_1$ and $D_2\subseteq P_2$, intersecting in an arc $\gamma$ and with arcs of intersection $\alpha_1 \subset \Sigma_1$ and $\alpha_2 \subset \Sigma_3$. Our  strategy is to reflect $D_2$ into the piece containing $D_1$, then show that $D_1$ can be isotoped to $D_2^R$ via an isotopy that can be extended to an isotopy of $A$. 
    
We formalize this with the notion of an \textit{extendable isotopy}, which is an isotopy of $D_1 \cup D_2^{R}$ that keeps $\gamma$ properly embedded in $\Sigma_2$ and keeps $\alpha_1$ and $\alpha_2^{R}$ properly embedded in $\Sigma_1$. 
We begin by classifying the possible intersections of $D_1$ and $D_2^R$.
\end{remark}

\begin{lemma}
    The intersections of $D_1 \cap D_2^R$ are either simple closed curves or arcs that fall into one of the three cases in Figure \ref{fig:selfIntersectionCases}.
\end{lemma}

\begin{figure}[htbp]
    \centering
    \includegraphics[width=1.0\linewidth]{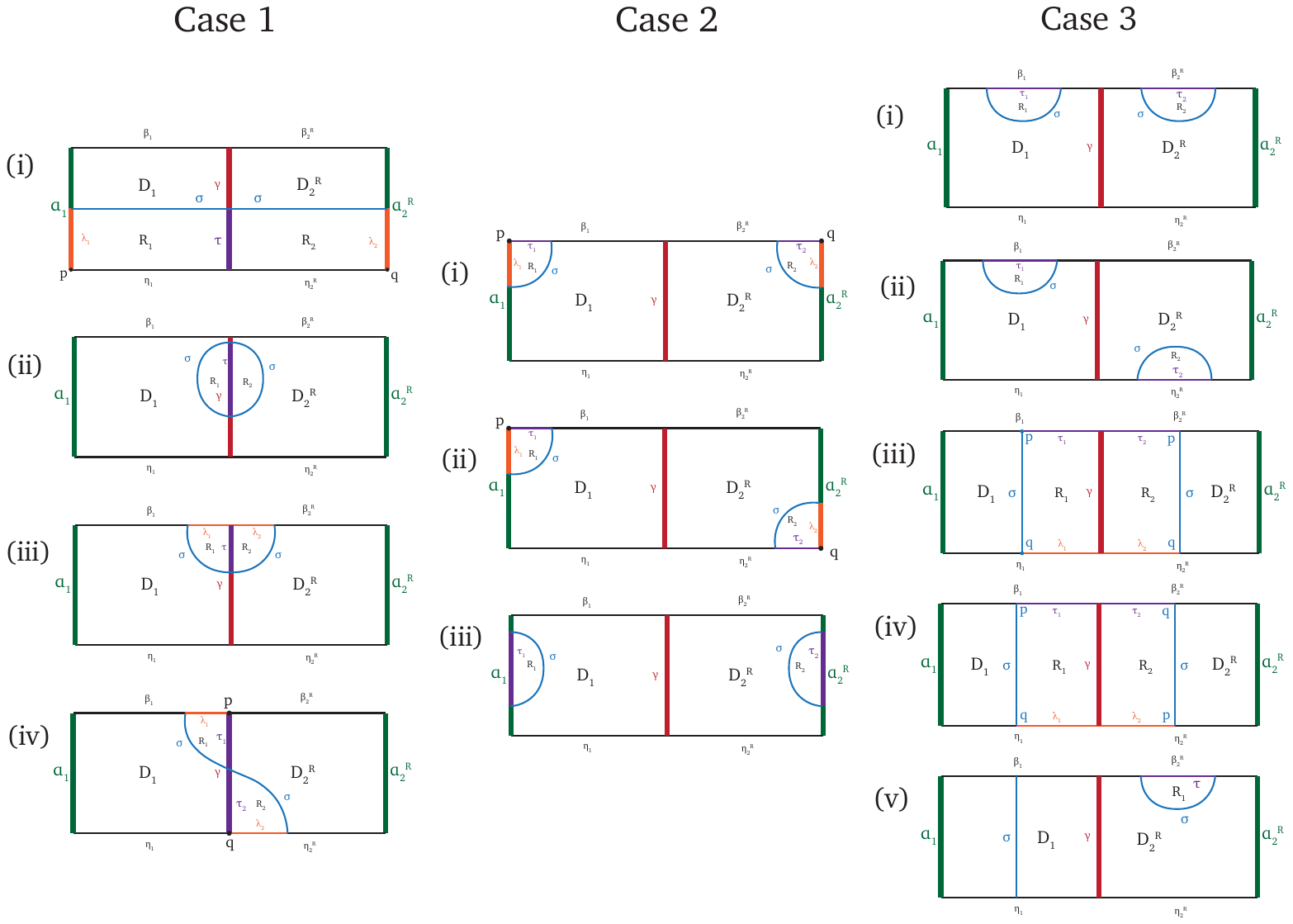}
    \caption{Possible intersections of $D_1$ and $D_2^{R}$.}
    \label{fig:selfIntersectionCases}
\end{figure}

\begin{proof}
We use the labels assigned in Figure \ref{fig:selfIntersectionCases}. We label arc endpoints when relevant. We classify intersection arcs into three cases as follows.

\begin{enumerate}
\item{ Arcs that have at least one endpoint on $\gamma$.} 
\item{ Arcs with no endpoints on $\gamma$ and that have at least one endpoint on $\alpha_1$.}
\item{ Arcs with no endpoints on $\gamma$ or $\alpha_1$.}
\end{enumerate}

Note $\gamma \subset (D_1 \cap D_2^{R})$. Thus if a curve has an endpoint on $\gamma$ in $D_1$, it must also have an endpoint on $\gamma$ in $D_2^{R}$. Note $\alpha_1$ and $\alpha_2^R$ both lie on $\Sigma_1$ while the $\beta_i$ and the $\eta_i$ intersect $\Sigma_1$ only on their endpoints. This implies any arc which intersects $\alpha_1$ on $D_1$ must intersect $\alpha_2^{R}$ on $D_2^{R}$. Thus, the classification for Case 1 is complete. The previous observation shows the classification for Case 2 is complete. Case 3 contains all remaining cases up to symmetry and is trivially complete.
\end{proof}

\begin{remark}
Any simple closed curve intersections on $D_1$, $D_2^{R}$ can be isotoped away by irreducibility of $D^{2n+2}(C)$. Starting with an innermost such intersection curve on $D_1$, we use the ball bounded by the corresponding disk $D_1'$ on $D_1$ and $D_2'$ on $D_2^R$ to isotope $D_1'$ to $D_2'$. Even if there are additional simple closed curves of intersection in $D_2'$, we can push the additional parts of $D_1$ out of the way during the isotopy.  Thus, we only consider the remaining non-trivial intersections.
\end{remark}

\begin{lemma}
Any intersection curve from Cases 1, 2 and 3 can be removed via an extendable isotopy of $D_1$ and $D_2^{R}$.
\label{lemma:case_2_intersection}
\end{lemma}

\begin{proof}
We use similar arguments to eliminate arcs of intersection. First, we eliminate Case 3(v). We can glue the disk $R_2$, cut off $D_2^R$ by $\alpha$ and chosen so it does not contain any other intersection curves, to the disk $R_1$ cut off from $D_1$ by $\alpha$, gluing them along $\alpha$ to obtain a $\partial$-compression disk for $A$, a contradiction. The same argument works if the vertical arc is on $D_2^R$ and the semicircular arc on $D_1$. So this case cannot occur.

Now, consider arcs of type Case 1(ii). Taking an innermost such, we have disks $R_1$ on $D_1$ and $R_2$ on $D_2^R$. Their union is a sphere in $D^{2n+2}(C)$ which therefore bounds a ball by irreducibility. Using this ball, we can isotope $R_1$ to $R_2$ and eliminate the intersection. 

In  Cases 3(i), (ii), we choose such an intersection that bounds a disk $R_1$ on $D_1$ that contains no other intersection curves, and the corresponding $R_2$ on $D_2^R$. Then, $R_1 \cup R_2$ is a properly embedded disk in $P_i$ with boundary in $\partial P_i \setminus \widehat{\Sigma}_i$.

Because $\D^{2n+2}(C)$ is $\partial$-irreducible,   $\partial R_1 \cup R_2$ must therefore bound a disk in $\partial H_{2n+2,C}$, call it $R_3.$ Then, $R_1 \cup R_2 \cup R_3$ is a sphere bounding a ball, which allows us to isotope $R_1$ to $R_2$ and eliminate the intersection. Note that in the process, we can push any parts of $D_1$ that are in the ball out of the way and eliminate additional intersection arcs. 

Once we have eliminated Cases 3(i) and 3(ii), we can use the same argument to eliminate Case 1(iii).

For Case 1(iv), taking $R_1$ to be innermost,  we have $R_1 \cup R_2$  a disk with one boundary arc running across $A$ from one boundary to the other and the other in 
$\partial P_i$. By $\partial$-incompressibility of $A$, this is a contradiction.

For Case 2(iii), again choosing $R_1$ with no other curves of intersection, $R_1 \cup R_2$ is a disk with boundary in $\widehat{\Sigma}_i$, which is incompressible. So, there is a disk $R_3$ in $\widehat{\Sigma}_i$ with the same boundary and $R_1 \cup R_2 \cup R_3$ is again a sphere bounding a ball, which allows us to isotope $R_1$ to $R_2$ and eliminate the intersection.

For Case 2(i), and  (ii), we again take $R_1$ innermost such on $D_1$ and $R_2$ the corresponding disk on $D_2^R$. Their union is a disk with boundary in $\partial P_i$, one arc of which is in $\widehat{\Sigma}_i$ and one in $\partial P_i \setminus \widehat{\Sigma}_i$. Because $(P_i, \widehat{\Sigma}_i)$ is an essential pair, $\widehat{\Sigma}_i$ is $\partial$-incompressible and therefore there is a disk $R_3$ in $\partial P_i$ with the same boundary. Hence, $R_1 \cup R_2 \cup R_3$ is a sphere that bounds a ball, allowing us to isotope $R_1$ to $R_2$ and eliminate the intersection.

In Case $1(i)$, choosing $R_1$ outermost, gluing sub-disk $R_1$ to $R_2$ along the arc $\sigma \cup \tau$ gives a disk $D$. Then, $\partial D$ is composed of two arcs, the first being $\lambda_1 \cup \lambda_2 \subset \Sigma_1$ and the second being $\eta_1 \cup \eta_2^R \subset \partial M$. Since $\Sigma_1$ is $\partial$-incompressible, there is an arc $\rho \subset \partial \Sigma_1$ from $p$ to $q$ co-bounding a disk $R_3$ with $\lambda_1 \cup \lambda_2$. Gluing $R_1 \cup R_2$ to $R_3$ along $\lambda_1 \cup \lambda_2$, yields a disk with boundary $\rho \cup \eta_1 \cup \eta_2 \subset \partial M$. By Lemma \ref{lemma:EssentialDisks}, there is some $R_4\subset \partial M$ with the same boundary. $R_1 \cup R_2 \cup R_3 \cup R_4$ is a sphere which bounds a ball, and we may isotope $R_1$ to $R_2$ while keeping $\alpha_1$, $\alpha_2^R$ on $R_3$, giving an extendable isotopy which eliminates the intersection as shown in Figure \ref{fig:case1iIntersectionNew}.
    
\begin{figure}[htbp]
        \centering
        \includegraphics[width=0.4\linewidth]{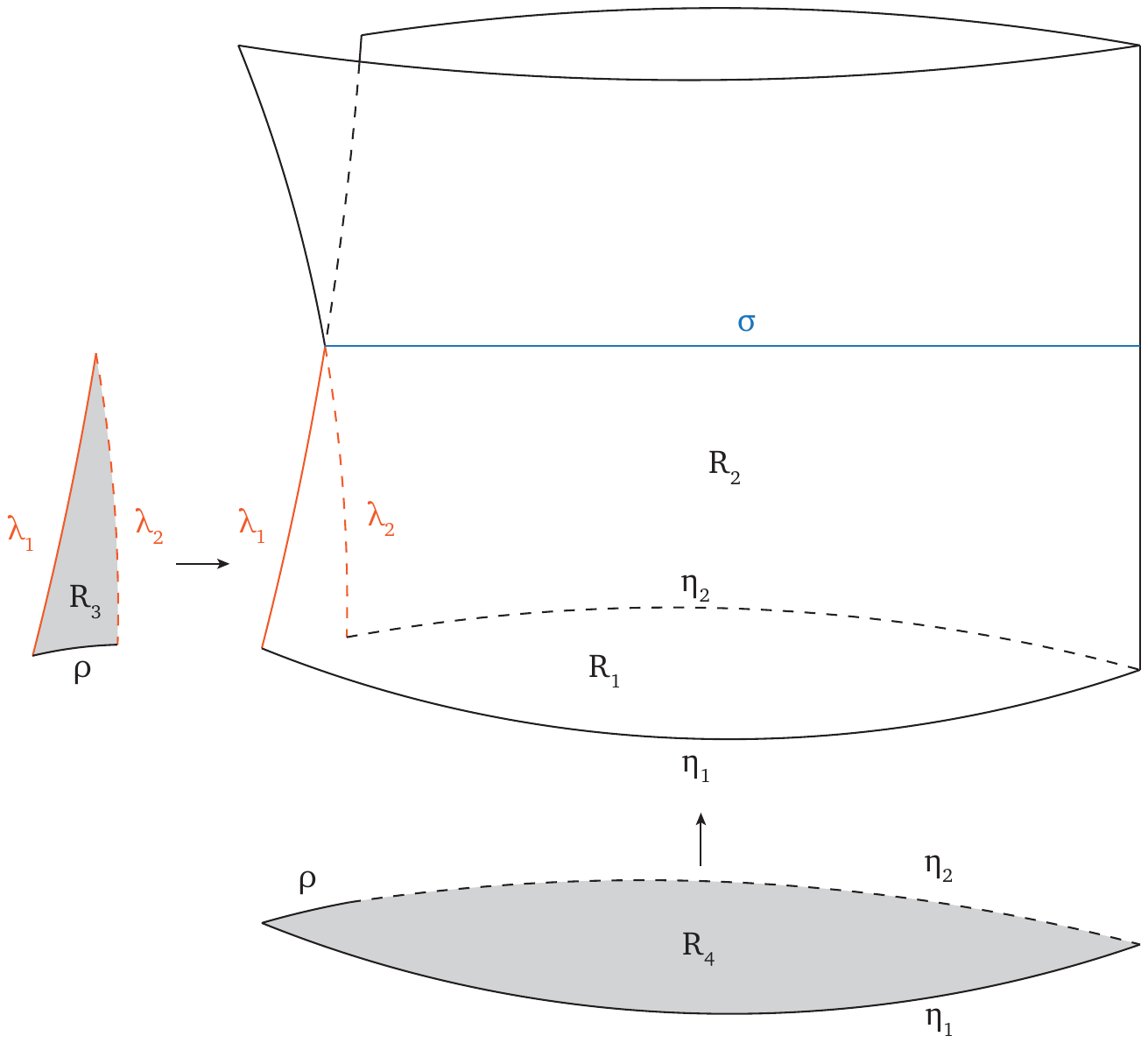}
        \caption{Push $R_1$ to $R_2$ through the ball.}
        \label{fig:case1iIntersectionNew}
    \end{figure}

    In Case $3(iii)$, gluing $R_1$ to $R_2$ along $\gamma$ and $\sigma$ yields an annulus $A'$ with boundary circles $\partial A_0' = \lambda_1 \cup \lambda_2$ and $\partial A_1' = \tau_1 \cup \tau_2$. The only sheet intersected by $A'$ is $\Sigma_2$, so by Lemma \ref{lemma:splitting_annulus}, $A'$ is compressible or $\partial$-compressible.

Note $\partial$-compressions for $A'$ yield $\partial$-compressions for $A$, so $A'$ is compressible. We proceed with a similar argument as in Lemma \ref{lemma:annulus_double_back}. Since $D^{2n+2}(C)$ is $\partial$-irreducible, there are compression disks $S_0$, $S_1 \subset \partial M$ with boundary $\partial A_0'$ and $\partial A_1'$ respectively. Since $\partial A_0' \cap \partial A_1' = \emptyset$, $S_0$ and $S_1$ are disjoint or concentric. If $S_0$ and $S_1$ are disjoint, $A' \cup S_0 \cup S_1$ forms a sphere which bounds a ball. We may isotope $R_1$ to $R_2$ while keeping $\lambda_1 \cup \lambda_2$ on $S_0$ and $\tau_1 \cup \tau_2$ on $S_1$, eliminating the intersection via an extendable isotopy. If $S_0 \subset S_1$, the boundary of a regular neighborhood of $(S_1 \cup A')$ forms a sphere which bounds a ball. We may isotope $R_1$ to $R_2$ through this ball, keeping the various arcs embedded on $S_1$ and $S_2$. See Figure \ref{fig:case3iiiIntersectionsNew}. 

\begin{figure}[htbp]
    \centering
    \includegraphics[width=0.4\linewidth]{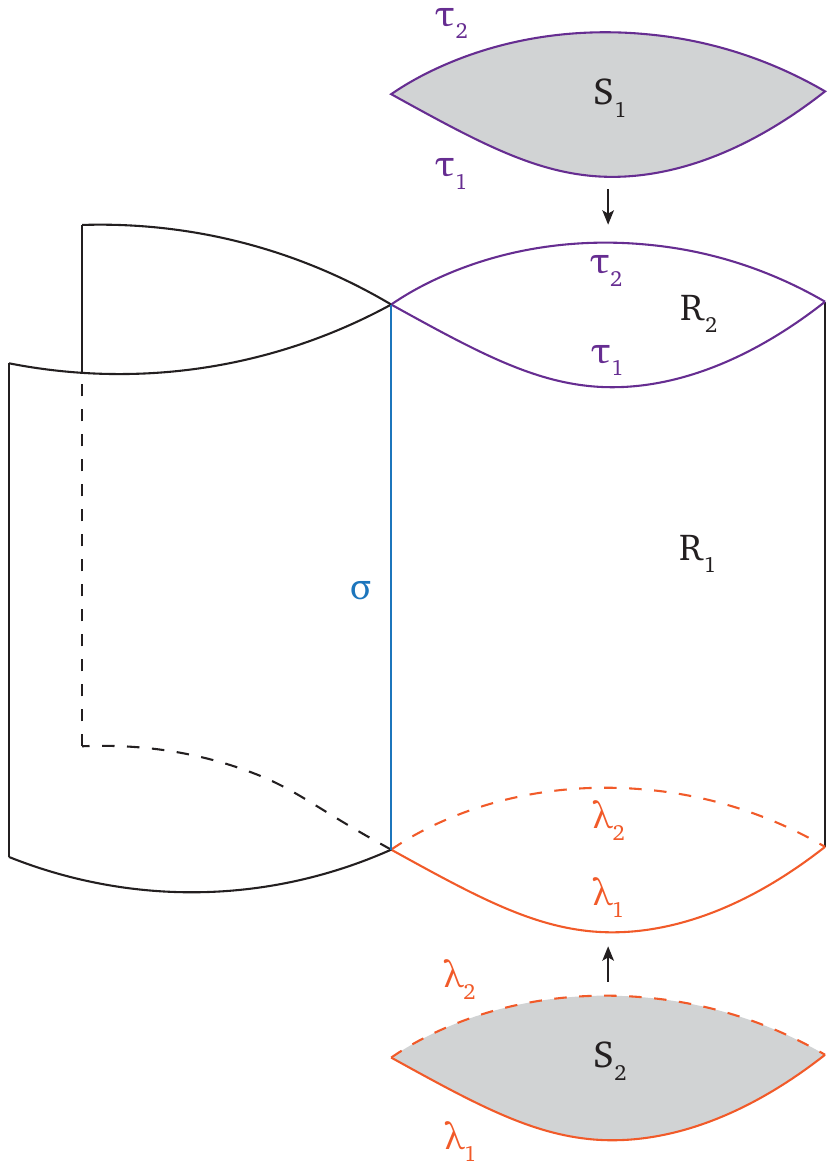}
    \caption{Compressions for $A'$ induce spheres which bound balls.}
    \label{fig:case3iiiIntersectionsNew}
\end{figure}

In Case $3(iv)$, gluing $R_1$ to $R_2$ along $\gamma$ and $\sigma$ yields a properly embedded M\"obius band $S$ in a piece, with boundary circle $\tau_1 \cup \tau_2 \cup \lambda_1 \cup \lambda_2$. 
   The boundary of a regular neighborhood of $S$ is an annulus $S'$, with both boundaries on $\partial M$ parallel to $\partial S$. A boundary-compression for $S'$ would imply a boundary-compression for $A$. Thus, by Lemma \ref{lemma:splitting_annulus}, $S'$ is compressible. But, by $\partial$-irreducibility of $D^{2n+2}(C)$, there is a disk $D$ on $\partial M$ such that its boundary is a boundary of $S'$. We can extend that disk just slightly to $\partial S$, and now $S \cup D$ is a projective plane embedded in the manifold, which cannot occur as our manifold exists in $\mathbb{R}^3$. Thus Case $3(iv)$ intersections cannot occur.
\end{proof}

We are now ready to prove there are no essential Type I annuli when $G$ is a collection of disjoint nontrivial arcs running from one boundary to the other.

\begin{proof}[Proof of Lemma \ref{lemma:arcAnnulusInessential}]
We have shown that we can isotope $D_1$ off of $D_2^R$ in $P_1$ with the exception that they still share $\gamma$. We now show that we can isotope $D_1$ to $D_2^R$ via an extendable isotopy. 

   Let $D' = D_1 \cup D_2^R$, a disk in $P_1$ with boundary in $\partial P_1 \setminus \partial N(L)$. Isotope it just slightly so it no longer intersects $\Sigma_2$ along $\gamma$.  Note that two arcs of its boundary lie on $\Sigma_1$. Doubling it across $P_i \cap \Sigma_1$ yields a Type I  annulus $A^*$ in $P_1 \cup P_1^R$, in $D^{2n+2}(C)$, as in Figure \ref{fig:Isotope D1 to D2R}. Since it does not intersect $\Sigma_2$,  Lemma \ref{annulusmissingsheet} implies it cannot be essential. 

   \begin{figure}[htbp]
    \centering
    \includegraphics[width=0.4\linewidth]{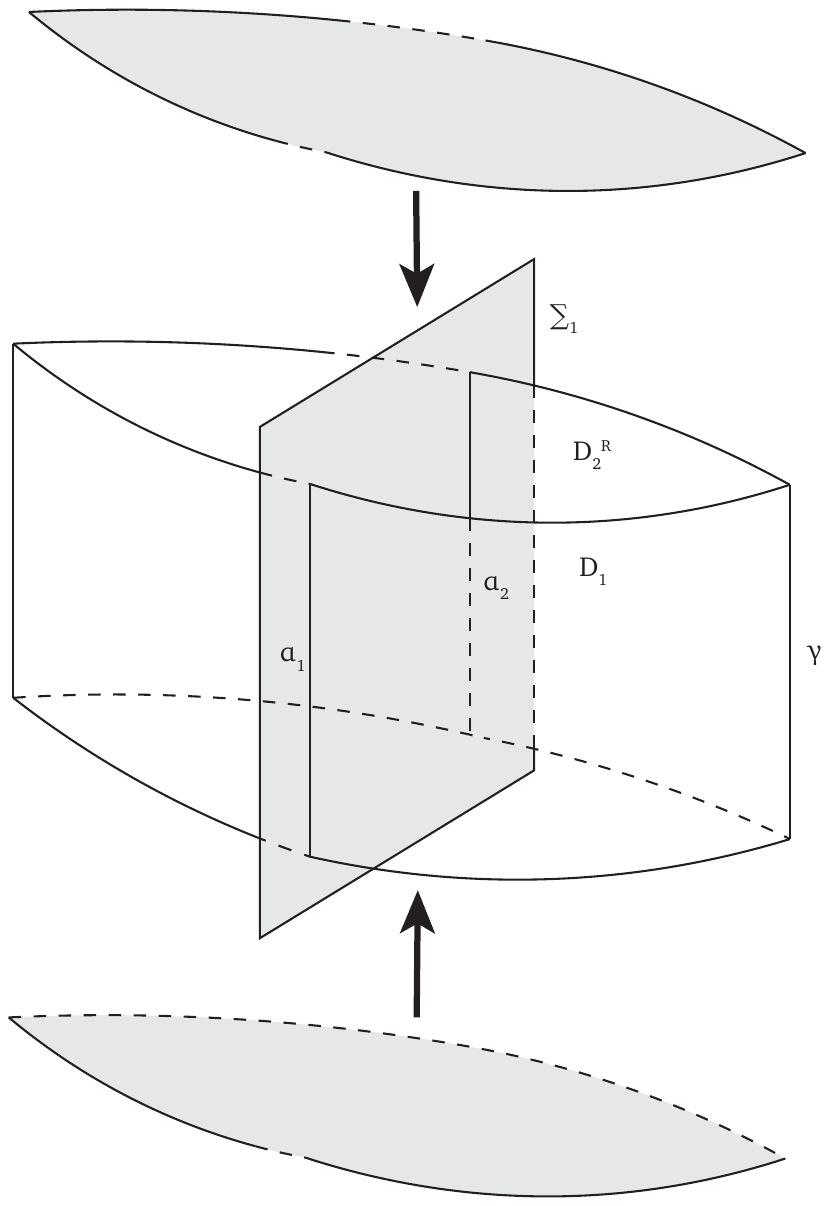}
    \caption{Once all intersections other than $\gamma$ have been eliminated, we reflect $D_1 \cup D_2^R$ to obtain an annulus that can be used to show $D_1$ can be isotoped to $D_2^R$.}
    \label{fig:Isotope D1 to D2R}
\end{figure}

If $A^*$ is boundary parallel, this would imply the same for $A$, a contradiction. Thus $A^*$ compresses along a disk $D^*$, and the union of $D^*$ with either of the two annuli it cuts $A^*$ into is a disk with boundary in $\partial H_{2n+2, C}$, which is $\partial$-irreducible and irreducible by Lemmas \ref{lemma:EssentialDisks} and \ref{lemma:EssentialSphere}. Hence each resulting disk has boundary that bounds a disk on $\partial H_{2n+2, C}$ and the resulting spheres bound balls in $D^{2n+2}(C).$ Gluing the balls together along $D^*$ yields a ball cut from $D^{2n+2}(C)$ by $A^*$. We can then use that ball to isotope $D_1$ to $D_2^R$ via an extendable isotopy, carrying along the rest of the original annulus $A$.

Since we have shown adjacent disks are extendably isotopic,  we may isotope $A$ so that as we travel cyclically along the pieces, $A \cap P_i$ alternates between $D_1$ and $D_1^{R}$.  Consider the annulus $A'$ formed by gluing alternating copies of $\phi(D_1)$ and $\phi(D_1^{R})$ in $D^{2n}(C)$. Note that since $D^{2n}(C)$ is hyperbolic, $A'$ is either compressible or $\partial$-compressible. 

Suppose $A'$ has a $\partial$-compression disk $D$. Recall $\partial D = \alpha \cup \beta$ with $\alpha \subset A'$ and $\beta \subset D^{2n}(C)$. We may isotope $D$ such that $\alpha$ lies in one piece. Following the argument outlined in Lemma \ref{lemma:splitting_bdy_incomp}, we may isotope $D$ such that it lies entirely in a piece. Then $\phi^{-1}(D)$ forms a $\partial$-compression disk for $A$, a contradiction.

Suppose $A'$ is compressible in $D^{2n}(C)$, with compression disk $D$. .
Then,  $\partial D$ splits $A'$ into two annuli. The union of $D$ with each annulus is a properly embedded disk with boundary in $\partial H_{2n, C}$, which itself is $\partial$-irreducible and irreducible. Hence, each resulting disk has boundary that bounds a disk on $\partial H_{2n, C}$ and the resulting spheres bound balls in $D^{2n}(C)$. Gluing the balls together along $D$ yields a ball cut from $D^{2n}(C)$ by $A'$. But $E$ intersects this ball in one of its two interval components, and any given $P_i$ intersects this ball in a sub-ball. Hence, the sub-ball exists in $C$ and reflecting it around when we create $D^{2n+2}(C)$ and $A$ within it shows that $A$ cuts a ball from $D^{2n+2}(C)$, a contradiction to $A$ being essential there. 
\end{proof}

\begin{lemma}\label{notypeI}
   There are no essential Type I annuli in $D^{2n+2}(C)$. 
\end{lemma}
\begin{proof} This follows from Lemmas \ref{lemma:splitting_annulus}, \ref{lemma:annulus_graph_lemma}, \ref{lemma:simpleClosedCurveAnnulusInessential}, and \ref{lemma:arcAnnulusInessential}. 
\end{proof} 

\begin{lemma}
   There are no essential Type II annuli in $D^{2n+2}(C)$. 
\end{lemma}
\begin{proof}
Let $A$ be such an annulus with boundary component $\partial A_1$  on $\partial N(K)$ for a  link component $K$ and $\partial A_0$ on $\partial M$.

The boundary of a regular neighborhood of $A \cup K$, yields a Type I annulus $A'$ as shown in Figure \ref{fig:oneBoundaryOnLink}. Note that this works for $\partial A_1$ any nontrivial $(p,q)$-curve on $\partial N(K)$.  Note also that $A'$ must be either compressible or $\partial$-compressible by Lemma \ref{notypeI}. If $A'$ is compressible, then $A$ must also be compressible. Thus, $A'$ has a $\partial$-compression disk $D$. Since the interior of the regular neighborhood contains $K$, $D$ must lie to the outside of the regular neighborhood. By Remark \ref{annuliinmanifolds}, we know that $A'$ is boundary-parallel and thus, cuts a solid torus from $D^{2n+2}(C)$. However, $A'$ also cuts a solid torus minus $K$ from its other side corresponding to the regular neighborhood of $A \cup K$. Hence, $D^{2n+2}(C)$ consists of the union of these two solid tori (one missing its core curve $K$) glued along essential annuli in their boundaries. But such a manifold has only torus boundaries, contradicting the fact $\partial H_{2n+2,C}$ has genus greater than 1. 

\end{proof}

\begin{figure}[htbp]
    \centering
    \includegraphics[scale=0.4]{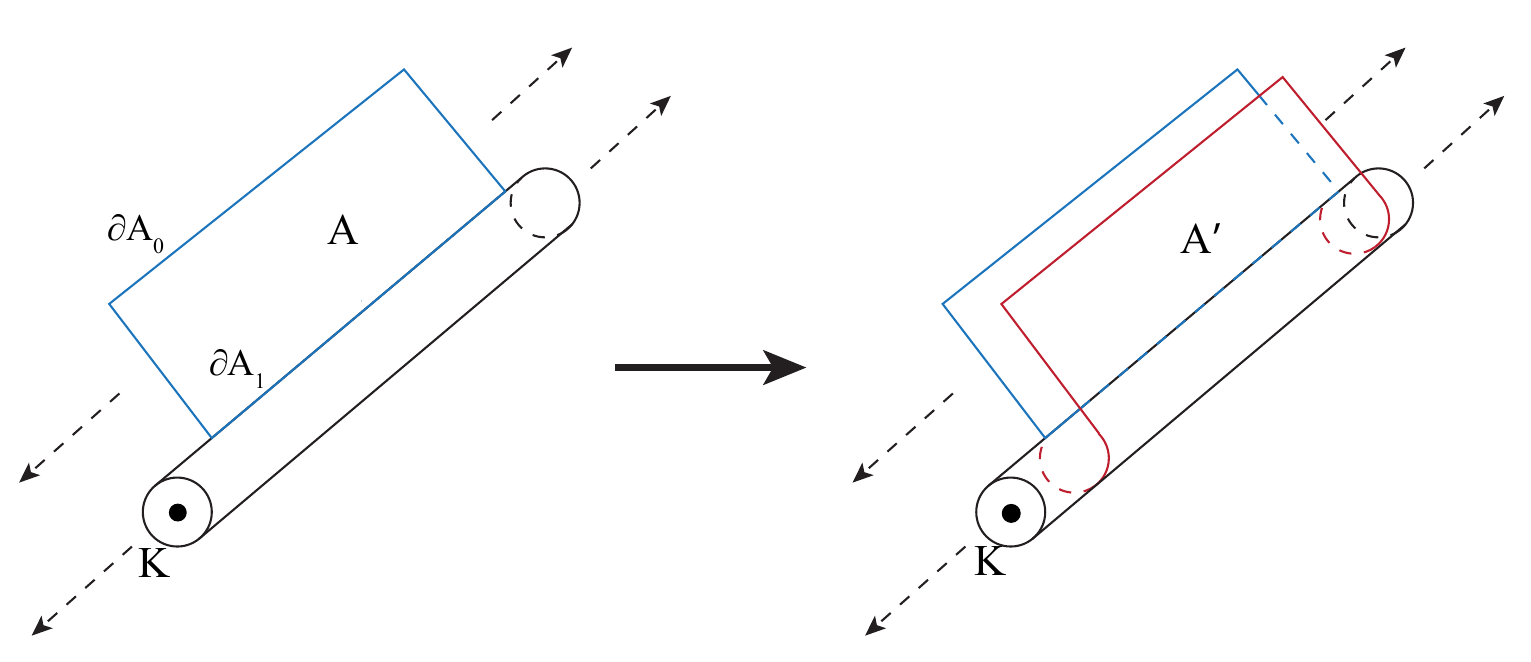}
    \caption{A Type II annulus generates a Type I annulus.}
    \label{fig:oneBoundaryOnLink}
\end{figure}

\begin{lemma}
   There are no Type III essential  annuli in $D^{2n+2}(C)$. 
\end{lemma}

\begin{proof}
Since $D^{2n+2}(C)$ is irreducible, $\partial$-irreducible and atoroidal, Lemma 1.18 of \cite{hatcher3manifold} implies  
the existence of such an annulus would force $D^{2n+2}(C)$ to be Seifert-fibered. Since $D^{2n+2}(C)$ contains at least four pieces and each piece contains at least one stake, $\partial M$ is not a torus. But Seifert fibered manifolds only have torus boundary components, so $D^{(2n+2)}(C)$ is not Seifert fibered, a contradiction.
\end{proof}

\begin{proof}[Proof of Theorem \ref{thrm:general_replicant}]
By Lemma \ref{lemma:EssentialSphere} and Lemma \ref{lemma:EssentialDisks}, $D^{2n+2}(C)$ is irreducible and $\partial$-irreducible. By Lemma \ref{lemma:EssentialTorus}, $D^{2n+2}(C)$ is atoroidal. By the lemmas in this section,  $D^{2n+2}(C)$ contains no essential annuli of Type I, II or III. Thus, $D^{2n+2}(C)$ is a simple Haken manifold, implying it is hyperbolic and completing the proof of the theorem.
\end{proof}

\subsection{Proof of Theorem}
 We derive the main theorem as a corollary to Theorem \ref{thrm:general_replicant}. 
 
\begin{reptheorem}{cor:2_replicant}
    Let $C$ be a charm. If $D^{2}(C)$ is tg-hyperbolic then $D^{2m}(C)$ is tg-hyperbolic for all $m \geq 1$.
   \end{reptheorem}

\begin{proof}
By Theorem \ref{thrm:general_replicant}, it is enough to show $(P_i, \widehat{\Sigma}_i)$ is an essential pair. Note that for $n=1$, $\widehat{\Sigma}_1$ is the fixed point set of a reflection symmetry of $D^{2}(C)$ and thus totally geodesic. Since $D^{2}(C)$ is hyperbolic it follows $\widehat{\Sigma}_1$ is essential. Thus $(D^{2}(C), \widehat{\Sigma}_1)$ is an essential pair, implying $(P_i, \widehat{\Sigma}_i)$ is also an essential pair. 
\end{proof}

\bibliographystyle{plain}
\bibliography{bib}

\begin{thebibliography}{1}

\bibitem{knotoids}
Colin Adams, Alexandra Bonat, Maya Chande, Joye Chen, Maxwell Jiang, Zachary Romrell, Daniel Santiago, Benjamin Shapiro, and Dora Woodruff.
\newblock Generalised knotoids.
\newblock {\em Math. Proc. Cambridge Phil. Soc.}, 177(1):67--102, 2024.

\bibitem{adams2021lower}
Colin Adams, Michele Capovilla-Searle, Darin Li, Lily~Qiao Li, Jacob McErlean, Alexander Simons, Natalie Stewart, and Xiwen Wang.
\newblock Lower bounds on volumes of hyperbolic 3-manifolds via decomposition.
\newblock 2021.

\bibitem{adams2023hyperbolicity}
Colin Adams and Joye Chen.
\newblock Hyperbolicity of alternating links in thickened surfaces with boundary.
\newblock {\em to appear in Algebraic and Geometric Topology}, 2026.

\bibitem{Adamsproceedings}
Colin Adams, Francisco Gomez-Paz, Jiachen Kang, Lukas Krause, Gregory Li, Reyna Li, Chloe Marple, and Ziwei Tan.
\newblock Hyperbolicity and volume of bongles.
\newblock {\em to appear in Knot Theory and its Applications, Palermo Proceedings}, 2026.

\bibitem{AdamsKang}
Colin Adams and Jiachen Kang.
\newblock Hyperbolicity of non-alternating staked links.
\newblock {\em in preparation}, 2026.

\bibitem{SnapPy}
Marc Culler, Nathan~M. Dunfield, Matthias Goerner, and Jeffrey~R. Weeks.
\newblock Snap{P}y, a computer program for studying the geometry and topology of $3$-manifolds.
\newblock Available at \url{http://snappy.computop.org}.

\bibitem{kauffmanstaked}
J.R. Goldman and Louis Kauffman.
\newblock Knots, tangles, and electrical networks.
\newblock {\em Advances in Applied Mathematics}, 14:267–306, 09 1993.

\bibitem{hatcher3manifold}
Allen Hatcher.
\newblock {\em Notes on Basic 3 Manifold Topology}.
\newblock \url{https://pi.math.cornell.edu/~hatcher/3M/3Mfds.pdf}, 2023.

\end{thebibliography}

\end{document}